\documentclass[11pt,leqno]{amsart}
\usepackage{amsthm,amsfonts,amssymb,amsmath,oldgerm}
\usepackage{graphicx, xcolor}
\numberwithin{equation}{section}
\usepackage{fullpage}
\usepackage{color}
\usepackage{calrsfs}
\DeclareMathAlphabet{\pazocal}{OMS}{zplm}{m}{n}

\newcommand{\dt}{\frac{\Delta_\tau x}{\tau}}
\newcommand{\dtx}{\Delta_\tau x'}

\def\eps{\varepsilon }

\newcommand\R{\mathbb R}

\def\eps{\varepsilon}

\newcommand\br{\begin{remark}}
\newcommand\er{\end{remark}}
\newcommand\bp{\begin{pmatrix}}
\newcommand\ep{\end{pmatrix}}
\newcommand{\be}{\begin{equation}}
\newcommand{\ee}{\end{equation}}
\newcommand\ba{\begin{equation}\begin{aligned}}
\newcommand\ea{\end{aligned}\end{equation}}

\newcommand{\bap}{\begin{app}}
\newcommand{\eap}{\end{app}}
\newcommand{\begs}{\begin{exams}}
\newcommand{\eegs}{\end{exams}}
\newcommand{\beg}{\begin{example}}
\newcommand{\eeg}{\end{exaplem}}
\newcommand{\bpr}{\begin{proposition}}
\newcommand{\epr}{\end{proposition}}
\newcommand{\bt}{\begin{theorem}}
\newcommand{\et}{\end{theorem}}
\newcommand{\bc}{\begin{corollary}}
\newcommand{\ec}{\end{corollary}}
\newcommand{\bl}{\begin{lemma}}
\newcommand{\el}{\end{lemma}}
\newcommand{\bd}{\begin{definition}}
\newcommand{\ed}{\end{definition}}
\newcommand{\brs}{\begin{remarks}}
\newcommand{\ers}{\end{remarks}}

\newcommand{\rmi}{{\mathrm{i}}}

\newcommand{\const}{\text{\rm constant}}
\newcommand{\Id}{{\rm Id }}

\newcommand{\Range}{{\rm Range }}

\newcommand{\sgn}{\text{\rm sgn}}
\newtheorem{theorem}{Theorem}[section]
\newtheorem{proposition}[theorem]{Proposition}
\newtheorem{corollary}[theorem]{Corollary}
\newtheorem{lemma}[theorem]{Lemma}

\theoremstyle{remark}
\newtheorem{remark}[theorem]{Remark}
\theoremstyle{definition}
\newtheorem{definition}[theorem]{Definition}

\newtheorem{example}[theorem]{Example}

\newcommand\cG{{\mathcal G}}

\newcommand\bH{{\mathbb H}}

\newcommand{\beq}{\begin{equation}}
\newcommand{\eeq}{\end{equation}}

\title{
Instantaneous smoothing and exponential decay of solutions for a degenerate
evolution equation with application to Boltzmann's equation
}

\author{Fedor Nazarov}
\address{Kent State University, Kent, OH 44240}
\email{nazarov@math.kent.edu}
\thanks{Research of F.N. was partially supported
under NSF grant no. DMS-1600239}
\author{Kevin Zumbrun}
\address{Indiana University, Bloomington, IN 47405}
\email{kzumbrun@indiana.edu}
\thanks{Research of K.Z. was partially supported
under NSF grant no. DMS-0300487}

\begin{document}

\begin{abstract}
We establish an instantaneous smoothing property for decaying solutions on the half-line $(0,+\infty)$
of certain degenerate Hilbert space-valued evolution equations arising in kinetic theory, including in 
particular the steady Boltzmann equation.
Our results answer the two main open problems posed by Pogan and Zumbrun in their treatment of $H^1$ stable manifolds
of such equations,
showing that $L^2_{loc}$ solutions that remain sufficiently small in $L^\infty$ (i) decay exponentially, and
(ii) are $C^\infty$ for $t>0$, hence lie eventually in the $H^1$ stable manifold constructed by Pogan and Zumbrun.
\end{abstract}

\date{\today}
\maketitle

\section{Introduction}

The goal of this paper is to prove instantaneous smoothing and decay properties for $\bH$-valued solutions of
the evolution equation
\be\label{maineq}
(d/dt)(Ax)=-x + \cG(x),
\ee
where $\bH$ is a separable Hilbert space, $A:\bH\to \bH$ is a constant bounded self-adjoint operator,
$\cG:\bH\to \bH$ is an infinitely differentiable map with
\be\label{gprop}
\hbox{\rm $\cG(0)=0$, $\sup_{x\in \bH}\|D_x\cG\|\leq 1/4$, and 
$\sup_{x\in \bH}\|D_x^k \cG\|<+\infty $ for all $k$.}
\ee

Our a priori assumptions are just that $x:(0,+\infty)\to \bH$ belongs to ($\bH$-valued) $L^2_{loc}(0,+\infty)$ and $Ax\in H^1_{loc}(0,+\infty)$, with \eqref{maineq} holding as an equation in $L^2_{loc}(0,+\infty)$;
we call such a function a ``weak $L^2_{loc}$ solution'' of \eqref{maineq}.
Note that these conditions do not imply any a priori smoothness for $x$ itself because $A$ may be quite degenerate;
in particular, we have in mind the case considered for kinetic equations in \cite{PZ1,PZ2,Z1} 
of $A$ essentially singular, or {\it one-to-one but not boundedly invertible}.

These conditions seem quite natural to impose just to make sense of the differential equation \eqref{maineq}.  
However it should be noted that our study was motivated by some questions that were left open in \cite{PZ1} 
where another notion of solution, the so-called ``mild solution,'' was introduced.
We relate our results to \cite{PZ1} by showing that any mild solution is in fact a solution in our sense as well.

\br\label{Brmk}
The studies of \cite{PZ1,PZ2,Z1} concerned the situation $\cG(x)=B(x,x)$ with $B$ a bilinear map; 
we note that this satisfies the framework \eqref{maineq}--\eqref{gprop} for 
solutions with $\sup |x|$ sufficiently small.
This is relevant to kinetic (in particular, Boltzmann's) equations, as discussed in 
Section \ref{s:applications}.
\er

\br\label{gammarmk} 
The condition $\sup_{x\in \bH}\|D_x\cG\|\leq 1/4$ can be weakened to
$$
\sup_{x\in \bH}\|D_x\cG\|\leq \gamma
$$
for $\gamma<1$, with essentially no change in the arguments.\footnote{
	Namely, in the first step of the proof of Theorem \ref{mainthm},
	we may use $|\cG(x)|=|\cG(x)-\cG(0)|\leq \gamma |x|$ to obtain
	$-|x|^2 + \langle \cG(x),x\rangle \leq -(1-\gamma)|x|^2\lesssim -|x|^2$ in place of \eqref{first},
	and similarly in higher-derivative estimates throughout.}
The choice $\gamma=1/4$ is for expositional convenience.
\er

Our main result is the following.

\bt\label{mainthm}
For $A$ constant,  bounded, and self-adjoint, and $\cG\in C^\infty$ satisfying \eqref{gprop},
every weak $L^2_{loc}$ solution $x$ of \eqref{maineq} lying in $ L^\infty(0,+\infty)$ is, in fact, 
$C^\infty$ on $(0,+\infty)$, and decays exponentially with all derivatives as $t\to\infty$.  
Moreover, for all $t>0$ and $k\geq 0$,
		\be\label{est}
		|(d/dt)^k x(t)|\leq \begin{cases} C t^{-k}, & 0<t < 1,\\
			Ce^{-ct}, & t \geq 1,
		\end{cases}
		\ee
		where $c$ depends only on $\|A\|$, and $C$ on $\|x\|_{L^\infty(0,+\infty)}$, 
	$\|A\|$, $k$, and $\sup_{x\in \bH} \|D^j_x\cG\|$, $j=1,2,\dots, k+1$.
		\et

\br\label{armk}
For $\mathcal{G}\in C^r$, we obtain instantaneous $H^r_{loc}$/$C^{r-1}_{loc}$ smoothing by the same argument.
\er

\br
A similar approach gives instantaneous smoothing for $x\in L^2(0,+\infty)$,
with bounds $|(d/dt)^{k}x(t)|\leq Ct^{-1/2 -k}$, $t>0$,\footnote{
Specifically, substituting $L^2(0,+\infty)$ for $L^\infty(0,+\infty)$ in item 1 of the 
proof in Section \ref{s:smooth}, and using the Sobolev estimate in item 2 to bound $|x(t)|$ for $ t\leq 1$, 
gives $|x(\tau)|\leq Ct^{-1/2}$ for $\tau\geq t$, yielding  the result by Theorem \ref{mainthm}. }
\er

We believe that Theorem \ref{mainthm} may be of its own interest, but as far as the questions in \cite{PZ1} are
concerned, it implies, in the slightly different setting of Remark \ref{Brmk},
that $L^2_{loc}(0,+\infty)$ 
solutions remaining sufficiently small in sup norm must
by exponential decay of $\|x\|_{H^1(t,+\infty)}$
belong eventually to the $H^1$ stable manifold constructed by Pogan-Zumbrun,
which may be characterized \cite{PZ1} as the union of all trajectories $x$ with $\|x\|_{H^1(0,+\infty)}$ 
sufficiently small.
Moreover, it implies that this $H^1$ stable manifold, expressed as in \cite{PZ1} 
as the union of trajectories in $H^1(0,+\infty)$, lies in $C^\infty(t,+\infty)$ for any $t>0$.
We recall that the analysis of \cite{PZ1} was motivated by consideration of
the steady Boltzmann equation, which, as shown in \cite{MZ1}, 
may be put by an appropriate coordinate transformation into the framework \eqref{maineq} considered
here, with $\cG(x)=B(x,x)$, $B$ a bounded bilinear map.

The paper is organized as follows.
In Section \ref{s:smooth}, we prove the theorem itself.  The proof is rather elementary and the reader 
interested in the key ideas will lose almost nothing assuming that we deal with a finite-dimensional Hilbert space
$\bH$, understand all derivatives in the classical sense, and just aim at quantitative bounds independent of 
dimension, etc.
We will still clearly state all standard facts from the integration theory of Hilbert space-valued functions to make sure 
that everything works in the generality we need, but 
we refer the reader to, e.g., \cite{DU} for their proofs.

In Section \ref{s:connect}, we revisit equation \eqref{maineq} from the standpoint of existence and uniqueness,
in the process connecting with \cite{PZ1} and the notion of ``mild solution''.
Unlike Section \ref{s:smooth}, which can be read from scratch, Section \ref{s:connect} assumes of the reader at
least some familiarity with standard Fourier transform,
convolution, and Gagliardo-Nirenberg bounds, 
and, differently from the situation in Section \ref{s:smooth}, the issues of how exactly everything is defined and in what sense the equalities hold are central there.
Though it is for the most part self-contained, the reader of Section \ref{s:connect} will benefit in
Section \ref{s:relation} from familiarity with \cite{PZ1} and in Remark \ref{1-1rmk} will need it
to make full sense of the remark.

In Section \ref{s:applications},
we discuss implications of our results for the questions considered 
in \cite{PZ1}, especially as they concern Boltzmann's equation.  
Finally, in Section \ref{s:discussion}, we 
describe some perspectives and open problems.
We delegate the proofs of one standard and one ``semi-standard'' Sobolev type embedding 
theorem (Lemmas \ref{ldiff}--\ref{ldiff2})
to be used in Section \ref{s:smooth} to an appendix so as not to interrupt the flow of the main argument.  
This appendix can be read completely independently of the main text.
In a second appendix, we discuss integrability of Hilbert space-valued functions as it relates to the
variation of constants formula in Section \ref{s:connect}.
This appendix refers to but is not needed in the main text.
It may be skipped by the reader if desired.

\medskip
{\bf Acknowledgement:} Thanks to Benjamin Jaye for helpful discussions, and to Alin Pogan for 
several readings of the manuscript and helpful suggestions for the exposition.
Thanks also to the anonymous referees for their careful reading and helpful suggestions.

\section{Proof of the main theorem}\label{s:smooth}

We start with a technical lemma.

\bl\label{techlem}
Suppose that $F:(0,+\infty)\to \R$ is absolutely continuous, $f,g$ are nonnegative with $f\in L^1_{loc}(0,+\infty)$, $g$ measurable,
and $(d/dt)F\leq -F^2g+f$ a.e. on $(0,+\infty)$.  Then, for every $t> t'>0$,
		$$
F(t)\leq \int_{t'}^t f + \Big(\int_{t'}^t g\Big)^{-1}.
		$$
\el

\begin{proof}
Let $G(s):=F(s)-\int_{t'}^s f$, $s\geq t'$.  If there exists $s\in (t',t)$ such that $G(s)\leq 0$, then
$$
F(t)=F(s) + \int_s^tF'\leq \int_{t'}^sf + \int_s^t f = \int_{t'}^t f
$$
and we are done. Otherwise, $G>0$ on $(t',t)$, so $0< G\leq F$ and
$$
(d/dt)G \leq -F^2g \leq -G^2g,
$$
or, equivalently, $(d/dt)(1/G)\geq g$, hence $(1/G)(t)\geq \int_{t'}^t g$ and $G(t)\leq \Big( \int_{t'}^tg\Big)^{-1}$.
\end{proof}

\bc\label{techcor}
Suppose that $h, y:(0,+\infty)\to \bH$ are in $L^2(0,+\infty)$, $\langle h,y\rangle$ is
absolutely continuous, $f:(0,+\infty)\to \R $ is in $L^1(0,+\infty)$ and nonnegative, and
$(d/dt)\langle h,y\rangle \leq -|y|^2 +f$ a.e. on $(0,+\infty)$.
Then, for every $t>t'>0$, there holds
$ \langle h,y\rangle(t) \leq \int_{t'}^t f + \frac{ \int_{t'}^t  |h|^2}{(t-t')^2}$ and
$$
\int_t^{+\infty} |y|^2 \leq \int_{t'}^{+\infty} f + \frac{ \int_{t'}^{t}  |h|^2}{(t-t')^2}. 
$$
\ec

\begin{proof}
Note that $|y|^2 \geq \frac{ \langle h,y\rangle^2}{|h|^2+\eps}$ for all $\eps>0$ and
$\Big(\int_{t'}^t (1/(|h|^2+\eps))\Big)^{-1}\leq \frac{\int_{t'}^t (|h|^2 +\eps)}{(t-t')^2}$,
so Lemma \ref{techlem} with $F=\langle h,y\rangle$, $g=1/(|h|^2+\eps)$ implies 
$ \langle h,y\rangle(t) \leq \int_{t'}^t f + \frac{ \int_{t'}^t  (|h|^2+\eps)}{(t-t')^2} $ for all $\eps>0$, yielding
the first inequality in the limit as $\eps\to 0^+$.
Since there exists a sequence $t_k\to +\infty$ such that $\langle h,y\rangle (t_k)\to 0$, we get
$ \int_t^{t_k} |y|^2\leq -\langle h,y\rangle|_t^{t_k} + \int_t^{t_k} f. $
Taking the limit as $t_k\to +\infty$, we get the result.
\end{proof}

In order to be able to use this corollary, we need the following key observation.
\bl\label{2.1lem}
If $A$ is a bounded self-adjoint operator,
$x:(0,+\infty)\to \bH$ is in $L^2_{loc}(0,+\infty)$, and $Ax\in H^1_{loc}(0,+\infty)$, then
$\langle Ax, x\rangle$ is absolutely continuous, with
\be\label{key}
(d/dt)\langle Ax,x\rangle=2\langle (d/dt)(Ax),x\rangle.
\ee
\el

\begin{proof}
	Fix $\tau>0$ and consider the difference quotient
$$
	\begin{aligned}
		\frac{\Delta_\tau\langle Ax,x\rangle}{\tau}(t)&=
\frac{\langle \Delta_\tau (Ax)(t),x(t+\tau)\rangle}{\tau}+ \frac{\langle Ax(t),(\Delta_\tau x)(t)\rangle}{\tau}\\
	& = \frac{\langle \Delta_\tau (Ax)(t), x(t) + x(t+\tau)\rangle}{\tau},
	\end{aligned}
$$
where $\Delta_\tau v(t):=v(t+\tau)-v(t)$.  
Here, we have used in a crucial way self-adjointness of $A$.
	Since $A$ is bounded, 
$\frac{\Delta_\tau \langle Ax, x\rangle}{\tau}\in L^1_{loc}(0,+\infty)$, and we have the integral identity
$$
\int_{t'}^t \frac{\Delta_\tau\langle Ax,x\rangle}{\tau}=
\big(S_\tau\langle Ax,x\rangle\big)\Big|_{t'}^t,
$$
where 
\be\label{Stau}
\big(S_\tau v\big)(t):=(1/\tau)\int_t^{t+\tau} v.
\ee

Note now that, as $\tau\to 0^+$, $S_\tau(\langle Ax,x\rangle)\to \langle Ax,x\rangle$ a.e. while
$$
x(t)+x(t+\tau)\to 2x(t) 
$$
in $L^2_{loc}$ and,  
by Lemma \ref{ldiff}(ii), $ \frac{\Delta_\tau (Ax)}{\tau} \to (d/dt)(Ax)$ in $L^2_{loc}$.
Hence, passing to the limit as $\tau\to 0^+$, we get
$$
\langle Ax,x\rangle|_{t'}^t=\int_{t' }^t 2\langle (d/dt)(Ax),x\rangle
$$
for a.e. $t',t\in (0,+\infty)$, verifying \eqref{key}.  
\end{proof}

\begin{proof}[Proof of Theorem \ref{mainthm}]
We proceed by a series of steps bounding $x$ in successively higher norms.
The main idea of the argument may be seen in steps 1 and 2 showing instantaneous smoothing from $L^\infty$ to $H^1$.
Steps 3--5, showing higher regularity, proceed in similar but more complicated fashion.

\medskip

1. ({\it Proof that $x\in L^2(0,+\infty)$}).
By time and space rescaling, we can always assume without loss of generality that $\|A\|=1$ and
$\|x\|_{L^\infty(0,+\infty)} \leq 1$.  Now, we have
	\be\label{first}
(d/dt)\langle Ax,x\rangle= 2\langle d/dt(Ax),x\rangle=
-2|x|^2 + 2\langle \cG(x),x\rangle \leq -|x|^2,
\ee
so $\langle Ax,x\rangle$ is decreasing.  Moreover, since $|x|^2 \geq|\langle Ax,x\rangle|$, 
either $\langle Ax,x\rangle\to -\infty$ as $t\to +\infty$, which is impossible, since
$x\in L^\infty(0,+\infty)$ and $A$ is bounded, or else $\langle Ax,x\rangle\geq 0$ 
for all $t>0$, and $$\langle Ax,x\rangle(t)\leq e^{-t}\langle Ax,x\rangle(0)\leq e^{-t}.$$
For, otherwise, $(d/dt)\langle Ax,x\rangle\leq -|x|^2\leq \langle Ax,x\rangle$ would be 
eventually uniformly negative.
			
Thus, for any $T> t>0$, $ \int_t^T |x|^2\leq -\langle Ax,x\rangle|_t^{T}\leq \langle Ax,x\rangle(t) \leq e^{-t}.  $
Letting $T\to +\infty$, we get $\int_t^{+\infty}|x|^2\leq e^{-t}$ for all $t>0$.  
In particular, $x\in L^2(0,+\infty)$ and $\|x\|_{L^2(0,+\infty)}\leq 1$.

\medskip

2. ({\it Proof that $x\in H^1_{loc}(0,+\infty)$}).
It is enough to show that the difference ratios 
$ \frac{\Delta_\tau x}{\tau}$ are uniformly bounded in $L^2_{loc}(0,+\infty)$ as $\tau\to 0^+$.
To this end, write
\be\label{**}
(d/dt)\Big(A\dt\Big) =- \dt + \frac{\Delta_\tau \cG(x)}{\tau}.
\ee
			
Note that $A\dt \in H^1_{loc}(0,+\infty)$ for any fixed $\tau>0$.  So, we get
\be\label{forlater}
(d/dt)\Big\langle A \dt, \dt\Big\rangle
=-2\Big| \dt\Big|^2 +2\Big\langle \frac{\Delta_\tau \cG(x)}{\tau},\dt\Big\rangle
\leq -\Big| \dt\Big|^2,
\ee
because $\sup_{x\in \bH} \|D_x\cG\|\leq 1/4$, so $\big|\frac{\Delta_\tau \cG(x)}{\tau}\big|\leq \frac14 \big|\dt\big|$.
Note also that 
$$
\Big| A\dt\Big|= \Big| \frac{\Delta_\tau (Ax)}{\tau} \Big|
=\Big| S_\tau \big( d/dt(Ax)\big) \Big|= \big| S_\tau \big(-x+\cG(x)\big)\big| \leq 2  S_\tau |x|
$$
for all $\tau$, where $S_\tau$ is the averaging operator of \eqref{Stau}.

Thus, applying 
Corollary \ref{techcor}
with $h=A\dt$, $y=\dt$, $f=0$, we get
$$
\int_t^{+\infty} \Big|\dt\Big|^2\leq \frac{4\int_{t'}^t (S_\tau |x|)^2}{(t-t')^2}
$$
for any $t>t'>0$.
Letting $t'\to 0$ when $0<t<1$ and putting $t'=t-1$ when $t\geq 1$, then using the bounds
$(S_\tau|x|)^2\leq 1$ and 
$ \int_{t-1}^\infty (S_\tau |x|)^2\leq  \int_{t-1}^{+\infty}|x|^2 \leq ee^{-t} $ respectively, we get
$$
\int_t^{+\infty} \Big| \dt\Big|^2\leq \begin{cases}
\frac4t, & 0<t<1,\\ 4ee^{-t}, & t\geq 1.
\end{cases}
$$
Thus, $x\in H^1(t,+\infty)$ for any $t>0$ and the same bounds hold for $x'$
(Lemma \ref{ldiff}(i), Appendix \ref{s:sob}).
	 
Moreover, we can now estimate $x(t)$ for $t\geq 1$ by writing
$$
|x(t)|^2= 2\Big| \int_t^{+\infty} \langle x, x'\rangle\Big| \leq
2 \Big(\int_t^{+\infty} |x|^2\Big)^{1/2} \Big(\int_t^{+\infty} |x'|^2\Big)^{1/2} \leq Ce^{-t}.
$$
Finally, applying Lemma \ref{ldiff}(i)--(ii), and passing to the $L^2$-limit as $\tau\to 0^+$ in \eqref{**},
we see that $Ax'\in H^1_{loc}(0,+\infty)$ and
$$ 
(d/dt)(Ax')=-x' + D_x\cG(x').  
$$
Specifically, we first apply Lemma \ref{ldiff}(ii) to the $H^1_{loc}(0,+\infty)$ functions
$x$ and $\mathcal{G}(x)$ to see that $\tau^{-1}\Delta_\tau x\to x'$ and 
$\tau^{-1}\Delta_\tau \mathcal{G}(x)\to (d/dt)\mathcal{G}(x)=D_x\mathcal{G}(x')$ on the right-hand side of \eqref{**}.
But, this implies that the limit 
$$
\lim_{\tau\to 0^+}\|(d/dt)A \tau^{-1} \Delta_\tau  x\|_{L^2(t,+\infty)}=
\lim_{\tau\to 0^+}\|\tau^{-1} \Delta_\tau A x'\|_{L^2(t,+\infty)}
$$
of the $L^2(t,+\infty)$ norm of the left-hand side of \eqref{**} exists and is bounded for all $t>0$, 
hence $Ax'\in H^1_{loc}(0,+\infty)$ by Lemma \ref{ldiff}(i).
Here, we have used the fact just established that $x\in H^1_{loc}(0,+\infty)$ to rewrite
$(d/dt)A \tau^{-1} \Delta_\tau  x= \tau^{-1} \Delta_\tau A x'$.
Applying Lemma \ref{ldiff}(ii), we find, finally, that 
\be\label{1eq}
(d/dt)Ax'= \lim_{\tau\to 0^+}
\tau^{-1} \Delta_\tau Ax'= -x' + D_x\mathcal{G}(x').
\ee

\medskip

3. ({\it Proof that $x\in W^{1,4}_{loc}(0,+\infty)$}).
We shall start by fixing $\tau>0$ and considering the difference $\dtx$.
We have $A\dtx \in H^1(t,+\infty)$ for any $t>0$ and, applying the linear 
ifference operator
$\Delta_\tau$ to \eqref{1eq} and using the product rule 
$\Delta_\tau yz= y (\Delta_\tau z) + \Delta_\tau y (z(\cdot+\tau))$:
\ba\label{L}
	(d/dt)(A\dtx)&=-\dtx + \Delta_\tau D_x \cG(x')\\
					  &= -\dtx +D_x \cG(\dtx) +  (\Delta_\tau(D_x\cG)) (x'(\cdot +\tau)).
\ea
Passing to the corresponding differential equation for the quadratic form $\langle A\dtx,\dtx\rangle$,
we can combine the first two terms on the right-hand side using the bound 
$$
|D_x\cG(\dtx)|\leq \frac14 |\dtx|
$$
to get
$$
\begin{aligned}
(d/dt)\langle (A\dtx), \dtx \rangle &\leq -\frac32 |\dtx|^2 + 2\langle \dtx, \Delta_\tau (D_x \cG)(x'(\cdot + \tau))\rangle\\
&\leq - |\dtx|^2 + 2\|\Delta_\tau D_x \cG\|^2 |x'(\cdot + \tau)|^2\\
&\leq - |\dtx|^2 + C|\Delta_\tau x|^2 |x'(\cdot + \tau)|^2,
\end{aligned}
$$
where $C$ is controlled by $\sup_{x\in \bH} \|D_x^2\cG\|$.
Here, we have used Young's inequality to bound the term
$2\langle \dtx, \Delta_\tau (D_x \cG)(x'(\cdot + \tau))\rangle$ by an absorbable term $(1/2)|\Delta_\tau x'|^2$
plus $2\|\Delta_\tau D_x \cG\|^2 |x'(\cdot + \tau)|^2$.
Using the equation, and the condition $\|D_x\cG\|\leq \frac14$, we also have
$$
|A(\dtx)|= |\Delta_\tau (d/dt)(Ax)|= |\Delta_\tau (-x + \cG(x))| \leq 2|\Delta_\tau x|.
$$
Finally, we have by Jensen's inequality 
$$
|\Delta_\tau x|^2(t)= |\tau S_\tau  x|^2(t)
\leq \tau^2 S_\tau (|x'|^2)(t)\leq \tau \int_t^{+\infty} |x'|^2
$$
for $S_\tau$ as in \eqref{Stau},
which is at most $\frac{C\tau}{t}$ for $0<t < 1$ and $C\tau e^{-t}$ for $t\geq 1$.

Applying 
Corollary \ref{techcor}
with $y=\dtx$, $h=A\dtx$, and $f= |\Delta_\tau x|^2 |x'(\cdot + \tau)|^2$, we get
\be\label{ineq}
\int_t^{+\infty} |\dtx|^2 \leq
\int_{t'}^{+\infty}|\Delta_\tau x|^2|x'(\cdot +\tau)|^2
+ \frac{ 4 \int_{t'}^t |\Delta_\tau x|^2}{(t-t')^2}.
\ee
Plugging in the estimates obtained 
just above 
for $|\Delta_\tau x|^2(t)$ and in step 2 for $\int_t^{+\infty}|x'|^2$, and letting
$t'=t/2$ for $0<t<1$ and $t'=t-1/2$ for $t\geq 1$, we get, noting that $\tau>0$ is arbitrary,
$$
\sup_{\tau>0}\int_t^{+\infty} \frac{ |\dtx|^2}{\tau} \leq
\begin{cases}
C t^{-2}, & 0<t<1,
\\ C e^{-t}, & t\geq 1.
\end{cases}
$$

By Lemma \ref{ldiff2}, Appendix \ref{s:sob}, this together with our previous estimate 
$$
\int_t^{+\infty} |x'|^2 \leq C/t
\quad \hbox{\rm for $0<t<1$}
$$
is enough to conclude
that $\int_t^{+\infty} |x'|^4 \leq   (1/2) (C/t^2(C/t)) =(1/2) (C^2/t^3)\lesssim t^{-3}$
for $0<t<1$. 
Similarly, using our previous estimate $\int_t^{+\infty} |x'|^2 \leq Ce^{-t}$ for $t\geq 1$, 
we obtain $\int_t^{+\infty} |x'|^4 \leq C^2e^{-2t}$ for $t\geq 1$.
Combining, we have for some (possibly larger) $C>0$
$$
\int_t^{+\infty} |x'|^4 \leq
\begin{cases}
C t^{-3}, & 0<t<1,\\
Ce^{-2t}, & t\geq 1.
\end{cases}
$$

\medskip

4. ({\it Proof that $x\in H^{2}_{loc}(0,+\infty)$}).
We shall use 
Corollary \ref{techcor}
%
once more with the same $y$, $f$, $h$, but this time we shall estimate the right-hand side 
of the inequality \eqref{ineq} in a different way.
First note that 
$$
\begin{aligned}
	|A(\dtx)|^2= |\Delta_\tau (Ax')|^2
&\leq \tau^2 S_\tau(\big|(d/dt)(Ax')\big|)^2\\
	&=
\tau^2 S_\tau(\big| -x' + D_x\cG(x')\big|)^2 \leq 4 \tau^2 S_\tau(|x'|)^2
\end{aligned}
$$
with $S_\tau$ as in \eqref{Stau}, so $\int_t^{+\infty} |A(\dtx)|^2\leq 4\tau^2\int_t^{+\infty} |x'|^2$.
Then observe that 
$$
\int_{t'}^{+\infty}|\Delta_\tau x|^2|x'(\cdot +\tau)|^2\leq 
\tau^2 \int_t^{+\infty} |S_\tau x'|^2|x'(\cdot +\tau)|^2
\leq \tau^2 \|x'\|_{L^4(t,+\infty)}^4.
$$

Thus, with the same choice $t'=t/2$ for $0<t<1$ and $t'=t-1/2$ for $t\geq 1$, we arrive at the bound
$$
\int_t^{+\infty} \frac{ |\dtx|^2}{\tau^2} \leq
\begin{cases}
C t^{-3}, & 0<t<1,\\ C e^{-t}, & t\geq 1.
\end{cases}
$$

The difference ratios $\frac{ \dtx}{\tau}$ are thus uniformly bounded in $L^2(t,+\infty)$ for any $t>0$, hence,
applying again Lemma \ref{ldiff}(i), Appendix \ref{s:sob}, we have
$x\in H^2_{loc}(0,+\infty)$ and
$$
\int_t^{+\infty} |x''|^2 \leq
\begin{cases}
C t^{-3}, & 0<t<1,\\ Ce^{-t}, & t\geq 1.
\end{cases}
$$
Also, 
$$
\begin{aligned}
	|x'(t)|^2\leq 2\Big|\int_t^{+\infty} \langle x',x''\rangle \Big| &\leq 2\Big(\int_t^{+\infty} |x'|^2\Big)^{1/2}
\Big(\int_t^{+\infty} |x''|^2\Big)^{1/2}\\
	& \leq
\begin{cases}
C t^{-2}, & 0<t<1,\\ Ce^{-t}, & t\geq 1,
\end{cases}
\end{aligned}
$$
verifying \eqref{est} for $k=1$.
Finally, applying again Lemma \ref{ldiff}(i)--(ii) and passing to the $L^2$-limit as $\tau\to 0^+$ on both sides of 
$\tau^{-1}$ times equation \eqref{L}, 
we get $Ax''\in H^1_{loc}(0,+\infty)$ and
\be\label{L2}
(d/dt)(Ax'')= -x'' + D^2_x\cG(x',x') + D_x\cG(x'').
\ee

Namely, we first observe that $x'$ and $(d/dt)\mathcal{G}(x)= D_x\mathcal{G}(x')$ are both contained in 
$H^1_{loc}(0,+\infty)$,
by the fact that $x'\in H^1_{loc}\cap L^\infty_{loc}(0,+\infty)$ together with the uniform derivative bounds
\eqref{gprop} on $\mathcal{G}$.
Thus, we may argue as in the end of step 2 to obtain first convergence in $L^2(t,+\infty)$ of $\tau^{-1}$ times
the right-hand side of \eqref{L} by Lemma \ref{ldiff}(ii), to 
$-x'' +(d/dt)D_x\mathcal{G}(x')= -x''+ D^2_x\cG(x',x') + D_x\cG(x'')$, for any $t>0$.
This implies convergence in $L^2(t,+\infty)$ of $\tau^{-1}$ times the left-hand side of \eqref{L}, or
$$
(d/dt)(\tau^{-1}\Delta_\tau x')= \tau^{-1}\Delta_\tau A x'',
$$
whence $Ax''\in H^1_{loc}(0,+\infty)$ 
by Lemma \ref{ldiff}(i).  
Applying Lemma \ref{ldiff}(ii), we obtain finally that $ \lim_{\tau\to 0^+}\tau^{-1}\Delta_\tau A x''= (d/dt)Ax''$,
yielding \eqref{L2} by equality of left- and right-hand limits.

\medskip

5. ({\it Proof that $x\in H^{J}_{loc}(0,+\infty)$, $J\geq 3$}). The rest of the argument we carry out by induction.
Specifically, at each level $J$, starting with $J=2$, we make the following induction hypotheses.

\medskip

(I1) For $0\leq k < J$, $x\in W_{loc}^{\infty, k}(0,+\infty)$, with $(d/dt)^k x$ satisfying
\eqref{est} for $t>0$.

(I2) For $1\leq k\leq J$, $x\in H^k_{loc}(0,+\infty)$, with 
\be\label{Hest}
\int_t^{+\infty} |(d/dt)^k x|^2  \leq
\begin{cases}
Ct^{-2k  +1}, & 0<t<1,\\
Ce^{-t}, &  t\geq 1.
\end{cases}
\ee

(I3) For $1\leq k\leq J$, $A(d/dt)^{k}x  \in H^1_{loc}(0,+\infty)$ and
$x$, $\mathcal{G}(x) \in H^k_{loc}(0,+\infty)$, with
\be\label{keq}
(d/dt) A (d/dt)^{k} x= - (d/dt)^k x + (d/dt)^k \mathcal{G}(x) \quad \hbox{\rm in $L^2_{loc}(0,+\infty)$.}
\ee

\medskip
We have shown in Step 4 that (I1)--(I3) are satisfied for $J=2$, i.e., $x\in W^{\infty, 1}_{loc}(0,+\infty)$
and \eqref{est} holds for $k=0, 1$; $x\in H^2_{loc}(0,+\infty)$, and
\eqref{Hest} holds for $k=1,2$; 
and $x$, $\mathcal{G}(x)$ are in $H^2_{loc}(0,+\infty)$, and $A(d/dt)^2 x$ is in $H^1_{loc}(0,+\infty)$, 
satisfying \eqref{keq} in the $L^2_{loc}(0,+\infty)$ sense for $k=0,1,2$.
We now show that satisfaction of (I1)--(I3) at level $J=j\geq 2$ implies satisfaction of (I1)--(I3)
at level $J=j+1$, whence, by induction, (I1)--(I3) hold for all $J\geq 2$.
This implies that $x$ is $C^\infty$ on $(0,+\infty)$ and satisfies \eqref{est} for all $k\geq 0$, completing the proof.
\medskip

By (I3), we have that $(d/dt) A\Delta_\tau (d/dt)^{j}x$, $\Delta_\tau (d/dt)^j x$, 
and $\Delta_\tau (d/dt)^{j}\mathcal{G}(x)$ are in $L^2_{loc}(0,+\infty)$, with
\be\label{A}
(d/dt) A \Delta_\tau (d/dt)^{j} x= - \Delta_\tau (d/dt)^j x +\Delta_\tau  (d/dt)^j \mathcal{G}(x).
\ee
Repeated application of the chain rule gives the expansion
\be\label{expansion}
(d/dt)^j \mathcal{G}(x)= \sum_{l=1}^{j}\sum_{s_1+ \cdots + s_l=j-l, \; s_j\geq 0} C^{j}_{l,s} \,
D_x^l \mathcal{G} \big((d/dt)^{s_1+1} x, \dots, (d/dt)^{s_l+1}x \big),
\ee
with $C^j_{l,s}$ denoting the number of occurences
of the derivative distribution $s=(s_1,\dots, s_l)$.
In particular, there is only one term for $l=1$, namely $C^j_{1,(j)}=1$.
Thus, we have
\ba\label{gform}
\Delta_\tau (d/dt)^j \mathcal{G}(x)&= 
\Delta_\tau \sum_{l=2}^{j}\sum_{s_1+ \cdots + s_l=j-l, \; s_j\geq 0} C^{j}_{l,s} \,
D_x^l \mathcal{G} \big((d/dt)^{s_1+1} x, \dots, (d/dt)^{s_l+1}x \big) \\
&\quad +  
(\Delta_\tau D_x\mathcal{G})  ((d/dt)^{j} x) + 
D_{x(\cdot +\tau)} \mathcal{G} (\Delta_\tau (d/dt)^{j} x).\\
\ea

Noting that 
\be\label{deltabd}
|D_{x(\cdot+\tau)} \mathcal{G} 
(\Delta_\tau (d/dt)^{j} x)| \leq (1/4)|\Delta_\tau (d/dt)^{j} x)|
\ee
by \eqref{gprop}, and
arguing as in step 3, we find therefore for $y=\Delta_\tau (d/dt)^j x$ and $h=A\Delta_\tau (d/dt)^j x$ that
$$
\langle y,h\rangle'\leq -|y|^2 + f,
$$
where $f=2|g|^2$ with $ g:= \Delta_\tau (d/dt)^{j} \mathcal{G}(x)- 
D_{x(\cdot+\tau)}\mathcal{G} (\Delta_\tau (d/dt)^{j} x))$.

With an eye toward applying Corollary \ref{techcor}, we first bound 
$\int_t^{+\infty} |f|=2 \int_t^{+\infty} |g|^2$. 
Evidently, $g$ is given by the sum of the first two terms on the right-hand side of \eqref{gform}.  The first, $\sum_{\ell=2}^j$ term,
involves only derivatives $(d/dt)^{s_i+1} x$ of order $<j$, since there are at least two summands $s_i+1$, and the sum of all such is $j$. 
Thus, we may use the relation $\Delta_\tau=\tau S_\tau (d/dt)$ to rewrite this term as
\be\label{Bsum}
\tau S_\tau \sum_{l=2}^{j+1}\sum_{s_1+ \cdots + s_l=j+1-l, \; s_j\geq 0} B^{j}_{l,s} \,
D_x^l \mathcal{G} \big((d/dt)^{s_1+1} x, \dots, (d/dt)^{s_l+1}x \big) 
\ee
with $B^{j}_{l,s}$ integer valued, 
where, in each summand, at most one of the derivatives $(d/dt)^{s_i+1}x$ is of order $j$, and the rest are of order $<j$.
Here, we are using in an important way the fact that $j\geq 2$. For $j=1$, the multi-index $s=(1,1)$ gives two derivatives of highest order $j$, 
a fact that cost some extra effort in steps 3--4.

By (I2)--(I3), therefore, the highest-order derivative appearing in each summand is bounded in $L^2(t,+\infty)$ by $Ct^{-(s_i+1)+1/2}$ for $0<t<1$ and $Ce^{-t/2}$ for $t\geq 1$, 
and the remaining derivatives are bounded in $L^\infty(t,+\infty)$ by $Ct^{-(s_i+1)}$ 
for $0<t<1$ and $Ce^{-t/2}$ for $t\geq 1$.
Combining these bounds with the uniform derivative bounds \eqref{gprop} on $\mathcal{G}$
and the fact that $S_\tau$ is bounded from $L^2(t,+\infty)$ to itself, 
we thus find that the $L^2(t,+\infty)$ norm of each summand in \eqref{Bsum} is bounded by
$$
\tau t^{-(\sum_{i=1}^l (s_i+1)) +1/2} =\tau t^{-(j+1) +1/2}
$$
for $0<t<1$ and $\tau e^{-t/2}$ for $t\geq 1$.

Likewise, in the remaining term 
$\big(\Delta_\tau D_x\mathcal{G}\big)((d/dt)^{j} x)$ of $g$, 
operator 
$$
\Delta_\tau D_x\mathcal{G}=\tau S_\tau (d/dt)D_x \mathcal{G}= \tau S_\tau D^2_x \mathcal{G}(x', \cdot)
$$
involves only derivatives of $x$ of order $1<j$, hence its operator norm may be estimated using the
$L^\infty(t,+\infty)$ bound \eqref{est} of induction hypothesis (I1), the bounds \eqref{gprop}, and boundedness 
in operator norm of the averaging operator $S_\tau$ as $ \|\Delta_\tau D_x\mathcal{G}\| \lesssim \tau t^{-1} $
for $0<t<1$ and $\lesssim \tau e^{-t/2}$ for $t\geq 1$.
Together with bounds 
$$\|(d/dt)^j)x\|_{L^2(t,+\infty)}\lesssim t^{-j+1/2} \quad \hbox{\rm for $0<t<1$}
$$
and
$\|(d/dt)^j)x\|_{L^2(t,+\infty)}\lesssim e^{-t}$ for $t\geq 1$
of induction hypothesis (I2), this gives 
\be\label{E}
\|\big(\Delta_\tau D_x\mathcal{G}\big)((d/dt)^{j} x)\|_{L^2(t,+\infty)} \lesssim 
\begin{cases} \tau t^{- 1}t^{-j+1/2} = \tau   t^{-(j+1)+1/2} 
	& \hbox{\rm for $0<t<1$,}\\
	\tau e^{-t/2} & \hbox{\rm  for $t\geq 1$}.
\end{cases}
\ee

Combining the above estimates, we obtain finally
\be\label{gbd}
\|g \|_{L^2(t,+\infty)} \lesssim 
\begin{cases} \tau t^{- 1}t^{-j+1/2} = \tau   t^{-(j+1)+1/2} 
	& \hbox{\rm for $0<t<1$ and }\\
	\tau e^{-t/2} & \hbox{\rm  for $t\geq 1$},
\end{cases}
\ee
and thus
$$
\int_t^{+\infty}|f| = 2 \int_t^{+\infty}|g|^2 \lesssim
\begin{cases}
	\tau^2  t^{-2(j+1)+1}, & 0<t<1,\\
	\tau^2 e^{-t}, & t\geq 1.
\end{cases}
$$

We next bound $\int_t^{+\infty} |h|^2$. 
Arguing as in previous steps, we have 
$$
h=A \Delta_\tau (d/dt)^jx = \tau S_\tau (d/dt) A(d/dt)^{j}x
= \tau S_\tau \big( -(d/dt)^j x + (d/dt)^j \mathcal{G}(x)\big).
$$
Expanding $(d/dt)^j\mathcal{G}(x)$ as in \eqref{expansion}
and noting that there appears in each term at most one derivative of $x$ of order $j$ and
the rest of order $<j$, we find that the $L^2(t,+\infty)$ norm of $(d/dt)^j\mathcal{G}(x)$ can be bounded using induction 
hypotheses (I1)--(I2)  by $Ct^{-j+1/2}$ for $0<t<1$ and $Ce^{-t/2}$ for $t\geq 1$.
Likewise, the $L^2(t,+\infty)$ norm of $(d/dt)^j x$ is bounded, by the induction hypothesis, by $Ct^{-j+1/2}$ for $0<t<1$ and $Ce^{-t/2}$ for $t\geq 1$.
Thus, noting again the harmless effect of averaging operator $S_\tau$, we obtain
$$
\int_{t}^{+\infty}|h|^2 \lesssim
\begin{cases}
	\tau^2  t^{-2j+1}, & 0<t<1,\\
	\tau^2 e^{-t}, & t\geq 1.
\end{cases}
$$

Applying Corollary \ref{techcor}
with $t'=t/2$ for $0<t<1$ and $t'=t-1/2$ for $t\geq 1$, we obtain therefore 
$
\int_t^{+\infty}|\Delta_\tau (d/dt)^j x|^2 \lesssim \tau^2 t^{-2(j+1)+1} 
$
for $0<t<1$ and $\lesssim  \tau^2 e^{-t}$ for $t\geq 1$, hence, by Lemma \ref{ldiff}(i),
$
\int_t^{+\infty}| (d/dt)^{j+1} x|^2 \lesssim   t^{-2(j+1)+1}
$
for $0<t<1$ and $\lesssim e^{-t}$ for $t\geq 1$, 
giving $x\in H^{j+1}_{loc}(0,+\infty)$ and verifying the $H^{k}$ bound of induction hypothesis (I2) for $k=j+1$.
This verifies (I2) for $J=j+1$.

Moreover, applying Lemma \ref{ldiff}(i)--(ii) as in the last parts of Steps 2 and 4, 
we obtain in the $L^2(t,+\infty)$ limit as $\tau\to 0^+$ of
\be\label{finaleq}
(d/dt) A \tau^{-1}\Delta_\tau   (d/dt)^{j} x= - \tau^{-1}\Delta_\tau (d/dt)^j x +\tau^{-1} \Delta_\tau (d/dt)^j \mathcal{G}(x)
\ee
that 
$A(d/dt)^{j+1} x\in H^1_{loc}(0,+\infty)$
and  $x$, $\mathcal{G}(x)\in H^{j+1}_{loc}(0,+\infty)$, with
\be\label{finaltarget}
(d/dt)A(d/dt)^{j+1} x= -(d/dt)^{j+1}x + (d/dt)^{j+1}\mathcal{G}(x),
\ee
verifying induction hypothesis (I3) at level $J=j+1$.

Specifically, recalling that $\Delta_\tau (d/dt)^j \mathcal{G}(x)= g + D_{x(\cdot+\tau)}\mathcal{G} (\Delta_\tau (d/dt)^{j} x))$,
we have from \eqref{deltabd}, \eqref{gbd}, and the previously-obtained bounds on $y=\Delta_\tau (d/dt)^{j} x))$
together with boundedness of $\mathcal{G}_x$ that
$$
\|\Delta_\tau (d/dt)^j \mathcal{G}(x)\|_{L^2_{loc}(t,+\infty)}\lesssim
\begin{cases} \tau t^{- 1}t^{-j+1/2} = \tau   t^{-(j+1)+1/2} 
	& \hbox{\rm for $0<t<1$ and }\\
	\tau e^{-t/2} & \hbox{\rm  for $t\geq 1$},
\end{cases}
$$
giving $\mathcal{G}(x)\in H^{j+1}_{loc}$, by Lemma \ref{ldiff}(i), and therefore
$$
\lim_{\tau\to 0^+} \tau^{-1} \Delta_\tau (d/dt)^j \mathcal{G}(x)= (d/dt)^{j+1}\mathcal{G}
$$
by Lemma \ref{ldiff}(ii).
By the fact $x\in H^{j+1}_{loc}$ verified in the previous step, we have by Lemma \ref{ldiff}(ii) also 
$$
\lim_{\tau\to 0^+} \tau^{-1} \Delta_\tau  (d/dt)^j x = -(d/dt)^{j+1}x .
$$
Thus, the right-hand side of \eqref{finaleq} converges to
$-(d/dt)^{j+1}x + (d/dt)^{j+1}\mathcal{G}(x)$ as $\tau\to 0^+$.

Using $x\in H^{j+1}_{loc}$ to re-express the left-hand side of \eqref{finaleq} as
$  \tau^{-1}\Delta_\tau (A   (d/dt)^{j+1} x)$, we obtain by 
Lemma \ref{ldiff}(i) that $  A   (d/dt)^{j+1} x \in H^1_{loc}$, and thus by 
Lemma \ref{ldiff}(ii), the left-hand side of 
\eqref{finaleq} converges to $(d/dt)A   (d/dt)^{j+1} x$.
Comparing the limits of left- and right-hand sides of \eqref{finaleq} then yields \eqref{finaltarget},
completing the verification of (I3) for $J=j+1$.

Finally, from our $H^{j+1}$ and $H^j$ bounds, we obtain 
the $L^\infty(t,+\infty)$ estimate \eqref{est} for $k=j$ by the one-dimensional Sobolev embedding estimate
$$
\begin{aligned}
	\|(d/dt)^j x\|_{L^\infty(t,+\infty)}^4 & \leq 
	\|(d/dt)^jx\|_{L^2(t,+\infty)}^2 \|(d/dt)^{j+1} x\|_{L^2(t,+\infty)}^2\\
	&\lesssim t^{-2j+1} t^{-2(j+1)+1}=t^{-4j}
\end{aligned}
$$
for $0<t<1$ and $\lesssim e^{-2t}$ for $t\geq 1$.
This verifies induction hypothesis (I1) for $J=j+1$, completing the induction and the proof.
		\end{proof}

\br\label{linrmk}
The instantaneous smoothing result \eqref{est} of Theorem \ref{mainthm} can be motivated by 
the canonical example of linear diagonal flow $Ax'=-x$ on $\bH=\ell^2$.
Writing $x=\{x_j\}$, with $\sum_j |x_j|^2<\infty$, take $Ax=\{a_j x_j\}$, with $a_j>0$ and
$a_j \to 0$ as $j\to \infty$.
Then, for $|x(0)|=1$, 
\be\label{linest}
|(d/dt)^k x(t)|\leq \sup_j a_j^{-k} e^{-a_j^{-1}t} |x(0)| \lesssim t^{-k},
\ee
in agreement with \eqref{est}.
Noting that $\max_{z\in \R^+} z^{-k} e^{-z^{-1}t} = e^{-1} t^{-k}$ is attained at $z=t$, 
we find by considering initial data of form $e_n=(0, \dots,0,  1, 0, \dots)$ with $n\to \infty$, 
and evaluating at $t=a_n$,
that \eqref{linest}, hence also \eqref{est},
is sharp in the sense that there is no better uniform bound as $t\to 0^+$.
\er

\section{Linear existence theory: weak vs. mild solutions}\label{s:connect}
We next compare our notion of ``weak $L^2_{loc}$'' solution with 
that of ``mild'' solution defined in \cite{PZ1}, 
in the process establishing linear existence and uniqueness of weak $L^2_{loc}$ solutions.
We show, first, that mild solutions are solutions in our sense as well,
hence subject to the smoothing results of section \ref{s:smooth}, and, second, that solutions $x\in L^2(t_0,t_1)$ 
in our sense are mild solutions in the sense of \cite{PZ1}
precisely if $\lim_{t\to \tau_0}|A|^{1/2}x(t) $ and
$\lim_{t\to \tau_1}|A|^{1/2}x(t) $ lie in $\Range \, |A|^{1/2}$.
Here, $|A|^{1/2}$ as we now describe is defined via the spectral decomposition formula 
for bounded self-adjoint operators.

Recall \cite{R} that a bounded self-adjoint operator $A$ on $\bH$ admits a spectral decomposition
\be\label{sd}
A = \int_\R \alpha dE_\alpha , \quad \Id = \int_{\R} dE_\alpha, 
\quad 
\langle x,y\rangle= \int_{\R} \langle x, dE_\alpha y\rangle,
\ee
where $dE_\alpha$ is a projection-valued measure.
Following the standard operator calculus, we define
$\sgn(A):= \int_\R \sgn(\alpha) dE_\alpha$,
$|A|= \sgn(A)A:= \int_\R |\alpha| dE_\alpha$, and
$$
\hbox{\rm $|A|^{r}:= \int_\R |\alpha|^{r} dE_\alpha$ for real $r>0$.}
$$

\subsection{The linear boundary-value problem and prescription of data for weak $L^2_{loc}$ solutions}\label{s:weaksoln}
The comparison of weak and mild solutions hinges ultimately on the question, of interest in its own right,
of how or in what sense boundary data is attained for weak solutions that are merely $L^2_{loc}$.

Formalizing the discussion of the introduction, we make the following definition.

\begin{definition}\label{d:weak}
	For $f\in L^2_{loc}$ on a given domain $D$, a \emph{weak $L^2_{loc}$ solution} of the linear
	inhomogeneous equation 
	\be\label{lin}
	(d/dt)(Ax)+ x=f
	\ee
	is a function $x(t)$ such that $x\in L^2_{loc}(D)$,  $Ax\in H^1_{loc}(D)$, and 
	\eqref{lin} holds on $L^2_{loc}(D)$.
\end{definition}

We begin by observing that a weak $L^2_{loc}$ solution,
or indeed any function with $x\in L^2_{loc}$
and $Ax\in H^1_{loc}$, admits a representative for which $|A|^{1/2}x$ is continuous.
Moreover, if $x, (d/dt)(Ax)\in L^2(\tau_0, \tau_1)$, then this representative
extends continuously to $[\tau_0,\tau_1]$.  
Thus, for solutions on $(t_0, t_1)$,
we may speak of the {\it boundary values of $|A|^{1/2}x$} at $t=\tau_0, \tau_1$.

\bl[Extension to $\R$]\label{extensionlem}
Let $x, (d/dt)(Ax) \in L^2(\tau_0, \tau_1)$, Then, there exists an extension $\widetilde x$ of $x$
to the whole line such that 
$$
\|\widetilde x\|_{L^2(\R)} + \|(d/dt)(A\widetilde x)\|_{L^2(\R)}\lesssim 
\|x\|_{L^2(\tau_0, \tau_1)} + \|(d/dt)(Ax)\|_{L^2(\tau_0, \tau_1)}.  
$$
\el

\begin{proof}
	We define $\widetilde x$ as the even reflection of $x$ across boundaries $t=\tau_0$ and $t=\tau_1$,
	supported on $[\tau_0-d, \tau_1+d]$, with $d=\tau_1-\tau_0$, that is:
	$$
	\widetilde x(t)=\begin{cases}
		x(\tau_0 + (\tau_0-t)) ,  & t\in [\tau_0-d,\tau_0]\\
		x(t),  & t\in [\tau_0,\tau_1]\\
		x(\tau_1- (t-\tau_1)) ,  & t\in [\tau_1, \tau_1+d].
		\end{cases}
	$$
	Clearly $\widetilde x$ is in $L^2$, with $L^2$ norm bounded by three times that of $x$.

	We show next that $A\widetilde x$ has a weak derivative $\widetilde y=(d/dt)(A\widetilde x)$ equal to the odd reflection
	of $y:=(d/dt)(A\widetilde x)$ across boundaries $t=\tau_0$ and $t=\tau_1$,
	likewise with $L^2$ norm bounded by three times that of $y$ on $(\tau_0,\tau_1)$:
	$$
	\widetilde y(t)=\begin{cases}
		-y(\tau_0 + (\tau_0-t)) ,  & t\in [\tau_0-d,\tau_0]\\
		y(t),  & t\in [\tau_0,\tau_1]\\
		-y(\tau_1- (t-\tau_1)) ,  & t\in [\tau_1, \tau_1+d],
		\end{cases}
	$$
	That is, we claim \cite{E} that 
	\be\label{weakdef}
	\int \langle (d/dt)\phi, A\widetilde x\rangle= -\int \langle \phi, \widetilde y\rangle
	\ee
	for all test functions $\phi\in C^\infty_0(\tau_0-d,\tau_1+d)$.

	For test functions $\phi$ that are even reflections about $t=\tau_1$ and supported in $(\tau_0, \tau_1+d)$,
	this follows because 
	$\langle (d/dt)\phi, A\widetilde x\rangle$ and $\langle \phi, y\rangle$  are both odd about $t=\tau_1$ as
	inner products of even and odd functions, hence have integral zero.
	For test functions $\phi\in C^\infty_0(\tau_0, \tau_1)$, it follows because the restrictions of 
	$\widetilde y$ and $\widetilde x$ to $(\tau_0,\tau_1)$ are $y=(d/dt)(Ax)$ and $x$. 
	Because test functions $\phi\in H^1_0(\tau_0, \tau_1)$ may be uniformly
	approximated in $H^1(\R)$ by test functions in $C^\infty_0(\tau_0,\tau_1)$, \eqref{weakdef} follows also
	for test functions $\phi\in H^1_0(\tau_0, \tau_1)$.
	By reflection, we find that \eqref{weakdef} holds also for test functions $\phi$ in 
	$H^1_0(\tau_0-d, \tau_0)$ or $H^1_0(\tau_1, \tau_1+d)$.


	But, any test function in $C^\infty_0 (\tau_0-d, \tau_1+d)$ may be decomposed into the sum of 	
	test functions that are even around $\tau_0$ and supported in $(\tau_0-d, \tau_1)$, even around
	$\tau_1$ and supported in $(\tau_0, \tau_1+d)$, plus test functions in $
	H^1_0(\tau_0-d, \tau_0)$, $H^1_0(\tau_0, \tau_1)$, and $H^1_0(\tau_1, \tau_1+d)$,
	whence \eqref{weakdef} holds by linear superposition for arbitrary $\phi\in C^\infty(\R)$ supported
	on $(\tau_0-d,\tau_1+d)$.
	Finally, we may multiply $\widetilde x$ by a smooth bump function that is identically equal to one
	on $[\tau_0, \tau_1]$
	and identically equal to zero outside $(\tau_0-d, \tau_1+d)$ 
	to obtain an extension of $x$ defined on all of $\R$ and satisfying the same bounds.
\end{proof}

\begin{corollary}\label{soblem}
If $x, (d/dt)(Ax) \in L^2(\tau_0,\tau_1)$, then $|A|^{1/2}x$ 
may be taken to be in 
$C^0[\tau_0, \tau_1]$, with 
	$$ \| |A|^{1/2}x\|_{C^0[\tau_0, \tau_1]}\lesssim  
	\|x\|_{L^2(\tau_0, \tau_1)} + \|(d/dt)(Ax)\|_{L^2(\tau_0, \tau_1)}.  $$
In particular, $|A|^{1/2}x$ may be taken to be continuous wherever $x$, $(d/dt)(Ax)\in L^2_{loc}$.
\end{corollary}

\begin{proof}
	Observing that the bounded linear operator $\sgn(A)$ commutes with weak differentiation, and applying
	Lemma \ref{extensionlem}, we may assume
	without loss of generality that $x, (d/dt)(|A|x)\in L^2(\R)$, with 
	$\|x\|_{L^2(\R)}\lesssim \|x\|_{L^2(\tau_0, \tau_1)}$ and 
	$$
	\|(d/dt)(|A|x)\|_{L^2(\R)}\lesssim 
	\|x\|_{L^2(\tau_0, \tau_1)} + \|(d/dt)(Ax)\|_{L^2(\tau_0, \tau_1)}.  
	$$

	Arguing as in Section \ref{s:smooth}, we have that 
		$ \langle |A|^{1/2}x, |A|^{1/2}x\rangle= \langle |A|x,x\rangle $
		is absolutely continuous and in $L^2(\R)$, with derivative  $2\langle (d/dt)(|A|x),x\rangle$, whence
		$$
		\langle |A|^{1/2}x, |A|^{1/2}x\rangle(t)= 2 \int_{-\infty}^t \langle (d/dt)(|A|x),x\rangle d\tau.  
		$$
Bounding $ \int_{-\infty}^t \langle (d/dt)(|A|x),x\rangle d\tau \leq \|(d/dt)(|A|x)\|_{L^2(\R)} \|x\|_{L^2(\R)} $
		 by Cauchy-Schwarz, and applying Young's inequality, we obtain 
	\ba\label{linftybd}
	\| |A|^{1/2}x\|_{L^\infty[\tau_0, \tau_1]}&\lesssim \|x\|_{L^2(\R)} + \|(d/dt)(Ax)\|_{L^2(\R)} \\
	& \lesssim \|x\|_{L^2(\tau_0, \tau_1)} + \|(d/dt)(Ax)\|_{L^2(\tau_0, \tau_1)}.
	\ea

	Continuity of $|A|^{1/2}x$ follows, finally, by a standard mollification argument, approximating $x$
	by $x^\eps:= x* \eta^\eps$, where $\eta^\eps(t)=\eps^{-1} \eta(t/\eps)$ is a smooth mollification kernel,
	$\eta\geq 0$ a $C^\infty$ bump function equal to $1$ for $|t|\leq 1/4$ and $0$ for $|t|\geq 1$, 
	with $\int_\R \eta(t)dt=1$.
	Noting that
	$x^\eps \to x$ and $(d/dt)(|A|x^\eps)\to (d/dt)(|A|x)$ in $L^2$ \cite{E}, 
	we find by \eqref{linftybd} applied to $|A|^{1/2}x^\eps$ that 
	$$
	\begin{aligned}
		\||A|^{1/2}x^{\eps_1}- |A|^{1/2}x^{\eps_2}\|_{L^\infty[\tau_0,\tau_1]}&\leq
	\|x^{\eps_1}-x^{\eps_2}\|_{L^2(\R)}\\
		&\quad + \|(d/dt)(Ax^{\eps_1})-(d/dt)(Ax^{\eps_2})\|_{L^2(\R)} \to 0
	\end{aligned}
	$$
	as $\eps_1,\eps_2\to 0$.
	Thus, the sequence $\{|A|^{1/2}x^\eps\}$
	is Cauchy in $L^\infty[\tau_0,\tau_1]$, whence $|A|^{1/2}x$ is a uniform limit
	of the continuous (indeed $C^\infty$) functions $|A|^{1/2}x^\eps$ as $\eps\to 0$, and thus continuous.
\end{proof}

We next recall two fundamental resolvent estimates from \cite{PZ1}.
In the remainder of this section, in order to use Fourier transform techniques,
we complexify the real Hilbert space $\bH$ in the standard way \cite[Ch. 1, Ex. 1.7]{C}.
as $\bH+ i\bH$ with inner product 
$$
\langle g_1+ig_2,f_1+if_2\rangle := \big(\langle g_1,f_1\rangle + \langle g_2,f_2\rangle\big)
+i 
\big(\langle g_1,f_2\rangle - \langle g_2,f_1\rangle\big).
$$

\bl[cf. Lemma 3.4, \cite{PZ1}]\label{estlem}
The Fourier symbol $(i\omega A+\Id)$, $\omega\in \R$ of $(A (d/dt) +\Id)$ satisfies
\be\label{mult}
\sup_{\omega \in \R} \|(i\omega A + \Id)^{-1}\| \; \leq 1, \quad
\sup_{\omega \in \R} \|i\omega A (i\omega A + \Id)^{-1}\| \; \leq 2.
\ee
\el

\begin{proof}
The first inequality follows by symmetry of $A$, which implies that the symmetric part of $(i\omega A+\Id)$ is just $\Id$, hence bounded below by $1$.
The second one then follows by resolvent identity  
$$
i\omega A(i\omega A + \Id)^{-1}= \Id - (i\omega A + \Id)^{-1}.
$$
\end{proof}

From \eqref{mult} and Parseval's identity, we find for $f\in L^2(\R)$ that $x$ defined by 
\be\label{sol}
\widehat x(\omega):=(i\omega A+\Id)^{-1}\widehat f(\omega)
\ee
$\widehat g$ denoting Fourier transform of $g$, 
gives a unique solution $x, (d/dt)(Ax) \in L^2(\R)$ of \eqref{lin} for $f\in L^2(\R)$, that is, a 
weak $L^2_{loc}$ solution in the sense of Definition \ref{d:weak} for the linear inhomogeneous problem on the line.

Now, define stable, center, and unstable projections 
$$
\begin{aligned}
	\Pi_s &:= \chi_{(0,+\infty)}(A)=\int_\R \chi_{\alpha>0}\, dE_\alpha,\;\\
	\Pi_c &:= \chi_{{0}}(A)=\int_\R \chi_{\alpha= 0}\, dE_\alpha,\;\\
	\Pi_u &:= \chi_{(-\infty,0)}(A)=\int_\R \chi_{\alpha< 0}\, dE_\alpha
\end{aligned}
$$
of $-A$, where $\chi_I$ denotes indicator function associated with set $I$, and stable, center, and ustable subspaces
$\Sigma_s=\Range \, \Pi_s$, $\Sigma_c=\Range \, \Pi_c$, and $\Sigma_u=\Range \, \Pi_u$.

It is straightforward to see that the operator-valued functions
\ba\label{solnop}
T_s(t) &:=  \int_\R \chi_{\alpha>0}\, e^{-t/\alpha} dE_\alpha, \quad t\geq 0 \\
T_u(t) &:=  \int_\R \chi_{\alpha<0}\, e^{-t/\alpha} dE_\alpha, \quad t\leq 0,
\ea
corresponding formally with $e^{-tA^{-1}}\Pi_s$ and $e^{-tA^{-1}}\Pi_u$, are strongly continuous
with $T_s(0)=\Id_{\Sigma_s}$ and $T_u(0)=\Id_{\Sigma_u}$,\footnote{
	Recall \cite{Pa} that strong continuity of an operator $T(t)$ on $\bH$ is defined as continuity of $T(t)x$ 
for each fixed $x\in \bH$.}
and for $h_s\in \Sigma_s$ and $h_u\in \Sigma_u$
generate solutions $x_s(t):=T_s(t)h_s$ and $x_u(t):=T_u(t)h_u$
of the homogeneous equation $( (d/dt)A +\Id)x=0$ in forward and backward
time, respectively, via $x_s(t):=T_s(t)h_s$ and $x_u(t):=T_s(t)h_u$.
For,
$$
\begin{aligned}
| (T_s(t+\delta)-T_s(t))h|^2 &=
	\int_\R \chi_{\alpha>0}\, (1- e^{-\delta/\alpha})^2 e^{-2(t+\delta)/\alpha}(h, dE_\alpha h)\\
	& \leq \int_\R \chi_{\alpha>0}\, (1- e^{-\delta/\alpha})^2 (h, dE_\alpha h) \to 0
\end{aligned}
$$
as $\delta\to 0^+$, for each fixed $h$, by Lebesgue Dominated Convergence, and similarly
for $-t< \delta \to 0^-$. A symmetric argument yields the result for $T_u$.

These are exactly the ``bi-stable semigroups'' constructed by quite different, Fourier transform 
means in \cite{PZ1} (cf. \cite[\S 2]{PZ1}).
Note that in general, $T_s$ is not bounded in the backward time direction, nor $T_u$ in forward time direction,
as $|\alpha|$ may be arbitrarily small, yielding arbitrarily large exponential growth $e^{|t/\alpha|}$.
In particular, the Cauchy problem  
$ ( (d/dt)A +\Id)x=0 $
for $t\gtrless 0$, $x(0)=x^0\in \bH$
is ill-posed in both forward and backward time directions.
Note also that the only homogeneous solutions on center subspace 
$\Sigma_c$ are, by inspection, the trivial ones $x(t) \equiv 0$.

For real $r>0$, define the unbounded operator $|A|^{-r}$ 
as the inverse of $|A|^{r}$ from $\Range  \,|A|^{r}$ to $\Sigma_s\oplus \Sigma_u$, that is,
$|A|^{-r}x=\int_{\R\setminus \{0\}} |\alpha|^{-r}dE_\alpha x$ for $x\in \Range  \,|A|^{r}$.
The next result shows that $|A|^{-1/2}T_s$ and $|A|^{-1/2}T_u$, 
give solution operators for boundary data $|A|^{1/2}x(0)$ in $\Sigma_s$ and $\Sigma_u$, respectively.

\begin{lemma}\label{Tbds} For $t>0$ and $t<0$, respectively,
	$T_s(t)$ and $T_u(t)$ take $\bH$ to $\Range  \,|A|^r$ for any $r>0$, with sharp bounds
	\be\label{invbds}
	\hbox{\rm $||A|^{-r}T_s(t)h|\leq Ct^{-r}|h|$ for $t>0$ and
	$||A|^{-r}T_u(t)h|\leq Ct^{-r}|h|$ for $t<0$.}
	\ee
	In particular, 
	$x_s(t):= |A|^{-1/2}T_s(t) g_s$ and $x_u(t):= |A|^{-1/2}T_u(t) g_u$ are well-defined for  
	any $g_s\in \Sigma_s$ and $g_u\in \Sigma_u$.
	Moreover, 
	\be\label{TAbd}
	\hbox{\rm $x_s$, $(d/dt)(Ax_s)\in L^2(\R^+)$ and $x_u$, $(d/dt)(Ax_u)\in L^2(\R^-)$,}
	\ee
	so that $x_s$ and $x_u$ solve $(d/dt)(Ax)+x=0$ in forward and backward time, respectively,
	with boundary values
	\be\label{bvals}
	\hbox{\rm $|A|^{1/2}x_s(0)=g_s$ and $|A|^{1/2}x_u(0)=g_u$,}
	\ee
	and these solutions are unique in the class $x, (d/dt)(Ax)\in L^2(0, \tau)$,
	for all $\tau >0$ and $\tau<0$, respectively.
\end{lemma}

\begin{proof}
	Noting that $\Range \, |A|^r$ consists of $x$ such that 
	$\int_{\R\setminus \{0\}} \langle x, |\alpha|^{-2r}dE_\alpha x\rangle<+\infty$,
	we obtain the first result from 
	$$
	\langle T_s(t)h, |\alpha|^{-2r}dE_\alpha T_s(t)h\rangle=
	\langle h, \chi_{\alpha>0}\, e^{-2t/\alpha}|\alpha|^{-2r}dE_\alpha h\rangle
	$$
	together with
	$
	|\alpha|^{-2r}|e^{-2t/\alpha} | \lesssim t^{-2r} 
	$
	for $t, \alpha>0$, as follows from $z^{2r}\lesssim e^{2z}$ for $z\in \R^+$.
	This gives at the same time \eqref{invbds}, which, by taking data $h$ with 
	measure $(h,dE_\alpha h)$ supported
	near $\alpha= t$, is easily be seen to be sharp.

	The second assertion, \eqref{TAbd}, follows similarly by the observation that 
	$$
	\int_0^{+\infty}(|\alpha|^{-1/2} e^{-t/\alpha})^2 dt
	=\int_0^{+\infty} |\alpha|^{-1}e^{-2t/\alpha}dt 
	= \int_0^{+\infty} e^{-2z}dz =\const,
	$$
	by substitution $z=t/\alpha$, $dz=\alpha^{-1}dt$.
	This gives by Fubini's Theorem 
	$$
	\begin{aligned}
		\| x_s \|_{L^2(\R^+)}^2&= 
		\int_{\R^+}\langle |A|^{-1/2}T_s g_0,|A|^{-1/2} T_s  g_0\rangle \\
		&=
		\int_{\R^+}  \langle g_0, |A|^{-1} T_s^2 g_0\rangle \\
		&=  \int_{\R^+} \int_{\R^+}  |\alpha|^{-1} e^{-2t/\alpha} \langle g_0, dE_\alpha g_0\rangle dt\\
		&= \int_{\R^+} \int_{\R^+} |\alpha|^{-1} e^{-2t/\alpha} dt  \langle g_0, dE_\alpha g_0\rangle \\
		&= C|g_0|^2,
	\end{aligned}
	$$
	hence $x_s\in L^2(\R^+)$ and, by $(d/dt)(Ax_s)=-x_s$, also $(d/dt)(Ax_s)\in L^2(\R^+)$.
	A similar computation gives \eqref{TAbd} for $T_u$.
	Meanwhile, \eqref{bvals} follows by $|A|^{1/2}x_s=T_s(t)g_s$,
	$|A|^{1/2}x_u=T_u(t)g_u$ and strong continuity of $T_s$, $T_u$ at $t=0^\pm$.

	Finally, uniqueness can be seen by an argument like that of Section \ref{s:smooth},
	after first projecting by  $\Pi_s$, $\Pi_u$, and $\Pi_c$ onto $\Sigma_s$, $\Sigma_u$, and $\Sigma_c$.
	For, $|A|^{1/2}x(0)=0$ implies 
	$$
	\langle x, Ax\rangle(0)=0.
	$$
	But, for solutions $x, (d/dt)(Ax)\in L^2$
	of homogeneous equation $(d/dt)(Ax)=-x$,
	the quadratic form $\langle x, Ax\rangle$ is absolutely continuous, with derivative $-\langle x,x\rangle\leq 0$.
	Restricted to $\Sigma_s$, where $A\geq 0$, we thus have by
	$$
	\langle x,Ax\rangle (t) = 
	\langle x,Ax\rangle (0) - \int_0^t |x|(s)^2 ds \leq 0
	$$
	that $\langle x,Ax\rangle (t) \equiv 0$ for $t \geq 0$. 
	This gives forward uniqueness, or uniqueness for $t\geq 0$, of the projection onto $\Sigma_s$.
	A similar argument yields backward uniqueness, or uniqueness for $t\leq 0 $, of the projection 
	onto $\Sigma_u$.
	Finally, on $\Sigma_c$, the homogeneous equation reduces to $x=0$, giving uniqueness for
	all $t$ of the projection onto $\Sigma_c$.

	Putting this information together, suppose we have a nontrivial solution $x\in (0,\tau)$ with
	$|A|^{1/2}x(0)=0$, $\tau>0$.  Then, the projection $x_s=\Pi_s x$ vanishes on $(0,\tau)$, as
	does the projection $x_c=\Pi_cx$. It remains to verify that $x_u=\Pi_u x$ vanishes on $(0,\tau)$.
	If $|A|^{1/2}x_u(\tau)=0$, then this follows by backward uniqueness on $\Sigma_u$. 
	If $|A|^{1/2}x_u(\tau)=g \neq 0 $, on the other hand, then 
	$x_u(t)\equiv |A|^{-1/2}T_u(t-\tau) g$, by backward uniqueness of functions valued
	in $\Sigma_u$.  In particular, we would have
	$$
	0=x_u(0)= |A|^{-1/2}T_u(-\tau) g
	= \int_\R \, \chi_{\alpha<0}\, |\alpha|^{-1/2} e^{\tau/\alpha} dE_\alpha g,
	$$
	which is evidently false unless $ dE_\alpha g\equiv 0$ and thus $g=0$.
	By contradiction, therefore, the result is proved.
\end{proof}

\br\label{backunique}
Note that we have obtained not only forward (backward) existence for Cauchy data in $\Sigma_s$ ($\Sigma_u$),
but uniqueness in both forward and backward directions.
For our purposes here, we only require forward (backward) uniqueness of solutions in $\Sigma_s$ ($\Sigma_u$);
however, the more general result seems interesting to note.
\er

We have also the following more familiar reinterpretation of the Fourier-transform solution \eqref{sol}
via variation of constants.
According to our earlier convention, define the unbounded operator $A^{-r}$ for integer $r>0$
as the inverse of $A^{r}$ from $\Range  \,A^{r}=\Range \, |A|^r$ to $\Sigma_s\oplus \Sigma_u$, that is,
$$
\hbox{\rm
$A^{-r}x=\int_{\R\setminus \{0\}} \alpha^{-r}dE_\alpha x$ for $x\in \Range  \,|A|^{r}$.}
$$
By Lemma \ref{Tbds}, $A^{-r}$ is well-defined on $T_s(t)x$ for $t>0$ and $T_u(t)x$ for $t<0$.

To state things most simply, define the spectral cutoffs 
$g^a(t):=\int_{|\alpha|\geq a} dE_\alpha g(t)$ for $a>0$ of a function $g\in L^2(\R)$.
Evidently, $g^a\to g_s+ g_u:=\Pi_s g + \Pi_u g$ both pointwise and in $L^2(\R)$ as $a\to 0^+$.
Then, we have the following variation of constants type formula, expressed in terms of improper integrals
with respect to the spectral parameter $\alpha$.

\bl\label{ylem}
The unique solution $y, (d/dt)(Ay)\in L^2(\R)$ of $(d/dt)(Ay)+y=f$ defined by \eqref{sol}
may be expressed alternatively as 
\ba\label{greens}
y(t)&=\lim_{a\to 0^+}\Big(\int_{-\infty}^t A^{-1}T_s(t-\tau)\Pi_s f^a(\tau)d\tau \\
	&\quad - 
	\int_t^{+\infty} A^{-1} T_u(t-\tau)\Pi_u f^a(\tau)d\tau \Big)  + \Pi_c f(t),
\ea
where the limit is taken in $L^2(\R)$.
In particular, for $f$ supported on $(\tau_0, \tau_1)$,
\be\label{cause}
\hbox{\rm 
$y_s(t)\equiv 0$ for $t\leq \tau_0$ and $y_u(t)\equiv 0$ for $t\geq \tau_1$.}
\ee
\el

\begin{proof}
Since $\Pi_s$, $\Pi_c$, and $\Pi_u$ commute with $A$, it is equivalent to show that 
$y_s:=\Pi_s y$, $y_u:=\Pi_u y$, and $y_c:=\Pi_c y$ are given by
$$
y_s(t)=\lim_{a\to 0^+}\int_{-\infty}^t A^{-1}T_s(t-\tau)\Pi_s f^a(\tau)d\tau, 
$$
$$
y_u(t)=- \lim_{a\to 0^+}\int_{-\infty}^t A^{-1}T_u(t-\tau)\Pi_u f^a(\tau)d\tau ,
$$
and $y_c=\Pi_c f$.
	
The third relation is nothing other than the projection of the evolution equation onto $\Sigma_c$, since
$\Pi_cA=0$.
	For operators $A$ with $|A|\geq a \Id>0$ in the sense of quadratic forms, 
	the first two follow from the standard variation
of constants formula for the solution of $(d/dt)x=-A^{-1}x$.
For, example, restricting for definiteness to the stable subspace $\Sigma_s$, we may first project
	the equation by $\Pi_s$ onto $\Sigma_s$, then, noting that $A$ in this case has bounded inverse,
	apply $A^{-1}$ to obtain the bounded-coefficient ODE
	\be\label{seq} (d/dt)y_s + A^{-1}y_s= (A^+)^{-1}\Pi_s f \ee
	Observing that $\|e^{-A^{-1}t}\Pi_s \|\leq Ce^{-t/\|A\|}$, we obtain by variation of constants
	that $y_s(t)=\int_{-\infty}^t A^{-1}T_s(t-\tau)\Pi_s f(\tau)d\tau $ 
	is the unique solution of \eqref{seq} in $L^2$.
	A symmetric argument yields the result for $y_u(t)$.

	Now, introduce the spectral cutoffs $f^a:=\int_{|\alpha|\geq a} dE_\alpha f$ of $f$ defined above, for $a>0$,
	and denote the corresponding solutions $y$ as $y^a$.
	By $f^a\to \Pi_s f+ \Pi_u f$ in $L^2(\R)$,
	together with the previously-shown boundedness of the solution operator, we have
	$y^a\to y_s+y_u$ in $L^2$.
	But, also, $|A|\geq a\Id>0$ with respect to functions supported on spectra $|\alpha|\geq a$,
	hence, by the discussion of the previous paragraph,
$$
	y^a(t)=\int_{-\infty}^t A^{-1}T_s(t-\tau)\Pi_s f^a(\tau)d\tau - 
\int_t^{+\infty} A^{-1} T_u(t-\tau)\Pi_u f^a(\tau)d\tau  
$$
	yielding \eqref{greens}. Property \eqref{cause} is an immediate consequence.
\end{proof}

\br\label{bochner_rmk}
The expression of \eqref{greens} in terms of improper integrals as $|\alpha|\to 0^+$ highlights again
the difference from the usual, nondegenerate case \cite{LP2}, for which the right-hand side of
\eqref{greens} may be expressed in terms of a standard Bochner integral on Hilbert space-valued functions \cite{DU}.
We show by explicit counterexample in Appendix \ref{s:bochner} that this is not necessarily the case
in the present, degenerate context.
%
\er

\br\label{varcrmk}
Equation \eqref{greens} differs from the standard variation of constants formula
in the final term $\Pi_c f(t)$, which appears to be of different form.
But, note that kernels $\alpha^{-1}e^{-t/\alpha}$ for the first two terms converge as $\alpha\to 0$
to a $\delta$-function, formally yielding the third upon convolution with $f$.
\er

Combining the above facts, we obtain the following solution formula for the {\it boundary-value problem on an interval},
that is, for solutions $x(t)$ in our $L^2_{loc}$ sense of
the linear problem \eqref{lin} on $(t_0,t_1)$
with boundary conditions imposed on the continuous image $|A|^{1/2}x(t)$ at $t=t_0, t_1$.

\begin{proposition}\label{excor} For each $f\in L^2(t_0,t_1)$, $g_0 \in \Sigma_s$, and $g_1 \in \Sigma_u$, 
	there is a unique weak $L^2_{loc}$ solution $x$, $(d/dt)(Ax) \in L^2(t_0,t_1)$,  of \eqref{lin}
	satisfying boundary conditions 
	\be\label{gBC}
(|A|^{1/2}\Pi_s x)(t_0)=g_0,\quad  (|A|^{1/2}\Pi_u x)(t_1)=g_1
\ee
on $|A|^{1/2}x$, given by
\ba\label{alt}
	x(t)&= y(t)+ |A|^{-1/2}T_s(t-t_0)g_0 + |A|^{-1/2}T_u(t-t_1)g_1,\\
	\widehat y(\omega)&= (i\omega A +\Id)^{-1}\widehat{f_{|[t_0,t_1]}}(\omega),  
	\ea
	or alternatively, with $y$ defined by \eqref{greens}, with $f$ extended as $0$ outside $(\tau_0,\tau_1)$.
\end{proposition}

\begin{proof}
	Evidently, $ z=x-y $
	is a solution of the homogeneous equation 
	$$
	(d/dt)(Az)+z=0
	$$
	satisfying 
	the prescribed boundary conditions, while $y$ is a solution of 
	$$
	(d/dt)(Ay)+y=f_{[t_0,t_1]},
	$$
	satisfying by \eqref{cause} homogeneous boundary conditions $|A|^{1/2}\Pi_s y(t_0)=0$ and 
	$$
	|A|^{1/2}\Pi_u y(t_1)=0.
	$$
	Thus, by linear superposition, $x=y+z$ satisfies \eqref{lin} with boundary conditions \eqref{gBC}.
	Uniqueness follows from uniqueness of solution $z_s$, $z_u$, and $z_c\equiv 0$
	of the homogeneous equation for $z$ on $[t_0,t_1]$ under the boundary conditions at $t_0$, $t_1$,
	a consequence of Lemma \ref{Tbds}.
\end{proof}

\subsection{ Relation to mild solutions}\label{s:relation} 
With these preparations, we are now ready to compare our notion of weak $L^2_{loc}$ solution
with that of ``mild'' $L^2_{loc}$ solution defined as follows in \cite{PZ1} (cf. \cite{LP2}).

\begin{definition}[\cite{PZ1,LP2}]\label{d:mild}
	For $f\in L^2_{loc}$ on a given domain $D$, a \emph{mild $L^2_{loc}$ solution} of \eqref{lin}
	is a function $x\in L^2_{loc}(D)$ with Fourier transform satisfying
	\ba\label{mild}
	\widehat{x}(\omega)&=
	 A(i\omega A+\Id)^{-1} \big[e^{-2\pi\rmi\omega t_0}h_0-e^{-2\pi\rmi\omega t_1}h_1\big]+
	 (i\omega A +\Id)^{-1}\widehat{f_{|[t_0,t_1]}}(\omega)
	\ea
with $h_0\in \Sigma_s$ and $h_1\in \Sigma_u$,\footnote{
	In \cite{PZ1,LP2}, for which there was assumed no center subspace, this definition was given
	for $h_j\in \bH$; however, without loss of generality it may be stated as above, thus allowing also the
	case of a nontrivial center subspace.
	}
or, equivalently  (see \cite[Remark 3.2]{PZ1}), the variation of constants type formula
\be\label{alt2}
x(t)= y(t)+ T_s(t-t_0)h_0 + T_u(t-t_1) h_1,
\quad \widehat y(\omega)= (i\omega A +\Id)^{-1}\widehat{f_{|[t_0,t_1]}}(\omega).  \ee
\end{definition}

Definition \ref{d:mild} gives a notion of a mild $L^2_{loc}$ solution of \eqref{lin}; 
a mild $L^2(\R^+)$ solution of \eqref{maineq} may then be defined as in \cite{PZ1,PZ2}
as a mild $L^2_{loc}$ solution $x\in L^2(\R^+)$ of \eqref{lin} with $f= \cG(x)$
(cf. \cite[Def. 3.1(ii)]{PZ1} and \cite[Lemma 3.3]{PZ1}). 

In \eqref{mild}--\eqref{alt2}, 
$h_0$ and $h_1$ are in general not connected with boundary values of $x$ at $t_0$ and $t_1$,
which may not even be defined. However, 
$$
T_s(t-t_0)h_0 + T_u(t-t_1) h_1
$$
is continuous, while $y(t)$ (since decaying at $\pm \infty$)
vanishes on $(-\infty, t_0)$ in stable modes and on $(t_1, \infty)$ in unstable modes.
Thus, for $H^1$ solutions, admitting continuous representatives (the class ultimately considered in 
the invariant manifold constructions of \cite{PZ1,PZ2,Z1}),
we have $\Pi_s x(t_0)=h_0$ and $\Pi_u x(t_1)=h_1$; in particular, for $H^1(\R^+)$ solutions of \eqref{maineq},
there is a well-defined boundary value $\Pi_s x(0)$ at $t=0$.

Comparing \eqref{alt} and \eqref{alt2}, and noting that 
$$
|A|^{-1/2}T_s |A|^{1/2}=T_s, \qquad |A|^{-1/2}T_u |A|^{1/2}=T_u, 
$$
we see that mild solutions are $L^2_{loc}$ solutions as defined here, but not conversely. 
Indeed, they are precisely the subclass of $L^2_{loc}$ solutions for which the boundary values $g_0$ and $g_1$
of $|A|^{1/2}x(t)$ at $t_0$ and $t_1$ are given by $|A|^{1/2}\Pi_s h_0$ and $|A|^{1/2}\Pi_u h_1$ 
with $h_j\in \bH$, i.e., for which $g_0$ and $g_1$ lie in $\Range \, |A|^{1/2}$.

\br\label{1-1rmk}
The representations \eqref{mild}--\eqref{alt2} were derived in \cite{PZ1} under the
assumption that $A$ be one-to-one; however, this assumption is not necessary, as shown by our analysis above. 
Indeed, one may check that the entire $H^1$ stable manifold construction of \cite{PZ1} goes through for general $A$.
We note that the key relation
$A(i\omega A+\Id)^{-1} e^{-2\pi\rmi\omega t_0}x(t_0) = \widehat{(T_s(t-t_0)x(t_0))}$
linking \eqref{mild} and \eqref{alt2} (see \cite[Remark 3.2(ii)]{PZ1}) 
follows in the general case from the result in the invertible case, by the observation that 
$$
A(i\omega A+\Id)^{-1} e^{-2\pi\rmi\omega t_0} =\Pi_{su} A(i\omega A+\Id)^{-1} e^{-2\pi\rmi\omega t_0},
$$
where $\Pi_{su}=\Pi_s+\Pi_u$ denotes projection onto $\Sigma_s\oplus \Sigma_u$.
Likewise, \eqref{alt} has the equivalent frequency-domain formulation
$$
\widehat{x}(\omega)=
A(i\omega A+\Id)^{-1}|A|^{-1/2} \big[e^{-2\pi\rmi\omega t_0}g_0-e^{-2\pi\rmi\omega t_1}g_1 \big]+
(i\omega A +\Id)^{-1}\widehat{f_{|[t_0,t_1]}}(\omega).
$$
\er

\section{Applications to Boltzmann's equation}\label{s:applications}
The main example considered in \cite{PZ1} was the {\it steady Boltzmann equation} (plus cousins and discrete approximations), 
\be\label{Boltz}
\xi_1 f_z=Q(f,f),\quad z\in \R^1, \, \xi\in \R^3,
\ee
with hard-sphere collision operator $Q$,
$f=f(z,\xi)$ denoting density at spatial point $z$ of particles with velocity $\xi$,
which, after the coordinate change
$f\to \langle \xi \rangle^{1/2}f$, $Q \to \langle \xi \rangle^{1/2}Q$, $\langle \xi\rangle:=\sqrt{1+|\xi|^2},$ 
can be put in form 
 $Aw_z = Q(w,w),$ 
with $\bH$ the standard square-root Maxwellian-weighted $L^2$ space in variable $\xi$,
$A=\xi_1/\langle \xi\rangle$ a multiplication operator, and $Q$ a bounded bilinear map \cite{MZ1}.
Note that $A$ has no kernel on $L^2$.
However, $0$ is in the essential range of the 
function $\xi_1/\langle\xi\rangle$, hence in the essential spectrum of the operator of
multiplication by $\xi_1/\langle\xi\rangle$.
That is, $A$ is ``essentially singular''.

Introducing the perturbation variable $x:=w-w_0$, where $w_0$ is an equilibrium, $Q(w_0,w_0)=0$,
and performing some straightforward further reductions \cite{PZ1,Z1} converts the equations 
to form \eqref{maineq}, with $t=z$ and $\cG(x)=B(x,x)$, $B$ a bounded bilinear map.
In this context, the problem considered here, of decay and smoothness of small solutions $x(t)$, amounts to the study of 
convergence and smoothness of solutions toward an equilibrium $w_0$. 
The particular motivation described in \cite{PZ1} was the desire to study existence and temporal stability
of large-amplitude heteroclinic connections, or planar {\it Boltzmann shock or boundary layers}, for which
the study of stable manifolds and decay to equilibria is an important first step.
The main result of \cite{PZ1} was construction of an $H^1$ stable manifold at $w_0$ containing all orbits 
sufficiently close to $w_0$ in $H^1(\R^+)$, exhibiting uniform exponential decay.

The global problem of existence and structure of large-amplitude Boltzmann shocks, 
as discussed by Truesdell, Ruggeri, Boillat, and others \cite{BR}, 
is one of the fundamental open problems in Boltzmann theory.
For this larger problem, it is important to know that the $H^1$ stable manifolds of \cite{PZ1} in fact
contain all candidates for heteroclinic connections, i.e., that the $H^1(0,+\infty)$ regularity imposed on solutions
in \cite{PZ1} is not too strong, eliminating potential connections.
Thus, the questions of regularity considered in the present paper are not just technical,
but central to the physical discussion.

In particular, we have answered here in the affirmative the two main open questions posed in \cite{PZ1}: 

1. Do asymptotically decaying (or just sufficiently small) $L^2_{loc}(\R^+)$ solutions of \eqref{maineq} decay exponentially in $t$?

2. Are small (in sup norm) $L^2_{loc}(\R^+)$ solutions necessarily in $H^1$ or higher Sobolev spaces?

\noindent
These results, together with those of \cite{PZ1}, imply that {\it the tail of any} (small- or large-amplitude)
{\it Boltzmann shock or boundary layer is $C^\infty$ and lies in the $H^1$ stable manifold constructed in \cite{PZ1}}.

\section{Discussion and open problems}\label{s:discussion}
The results of instantaneous smoothing obtained here are somewhat
analogous to interior regularity results for more standard boundary-value problems, e.g., elliptic and other boundary-value ODE.  
However, here, it should be noted, due to allowed degeneracy of $A$, there is in general no gain in regularity
in solutions of the linear inhomogeneous problem $(d/dt)Ax+ x=f$ of Proposition \ref{excor}, with $L^2$ forcing $f$ leading to
$L^2$ regularity and not higher of solutions $x$.
More, as discussed in \cite{PZ1,Z1}, the Fourier multiplier $(i\omega A+\Id)^{-1}$ is bounded on $L^p$ for $1<p<\infty$, but not $p=1$ or $\infty$;
thus, the solution operator is not associated with an integrable kernel as in more standard cases.

An interesting remaining open problem is to construct $L^\infty$ decaying solutions of \eqref{maineq} that are not small in $H^1$, i.e., backward extensions of the manifold of $H^1$ solutions constructed in \cite{PZ1}.
(Here, we showed that solutions that are eventually small in $L^\infty$ in fact decay to this manifold, 
but did not produce any such.)
A related very interesting open problem, moving toward construction of full heteroclinic orbits, is the question of backward uniqueness of solutions of \eqref{maineq}, 
i.e., whether $L^\infty(0,+\infty)$ solutions agreeing on $t\geq t'>0$ must agree on $t>0$.
Existence of large Boltzmann shocks- the ``structure problem of Ruggeri et al \cite{BR}-
is a major open problem, involving in addition separate, and presumably more problem-dependent, issues of global analysis.

We note that forward uniqueness of small $L^\infty(0,+\infty)$ solutions holds for $\mathcal{G}$ Lipschitz 
with small Lipschitz norm, by essentially the same argument as in step 1 of the proof of Theorem \ref{mainthm}
applied to the error equation governing $e:=x_1-x_2$, where $x_1$ and $x_2$ are two solutions with $x_1(0)=x_2(0)$.
This extends the result of forward uniqueness of
$H^1(0,+\infty)$ solutions following from the $H^1$ stable manifold results of \cite{PZ1}.
Thus, an interesting preliminary question is whether backward uniqueness of small $L^\infty(0,+\infty)$ solutions 
can fail for the same class of Lipschitz $\mathcal{G}$ with small Lipschitz norm.

\appendix
\section{Sobolev estimates for difference operators}\label{s:sob}

\bl\label{ldiff}
(i) For $f:(0,+\infty)\to \bH$ in $L^2(0,+\infty)$, there holds
$$
\sqrt{\int_0^\infty |f'|^2}
\lesssim \liminf_{\tau \to 0^+} \tau^{-1} \|f(\cdot +\tau)-f(\cdot)\|_{L^2(\R^+,\bH)}.  
$$
(ii) For $f:(0,+\infty)\to \bH$ in $H^1_{loc}(0,+\infty)$, there holds $\tau^{-1} (f(\cdot +\tau)-f(\cdot))\to f'$
in $L^2_{loc}$ as $\tau\to 0^+$.
\el

\begin{proof}
(i). By Fatou's Lemma,
	$$
	\begin{aligned}
		\|(d/dt) g\|_{L^2(\R)}&= \int_\R \omega^{2} |\widehat g|^2 d\omega\\
		&=\int_\R \lim_{\tau\to 0^+} |\tau^{-1} (e^{-i\tau \omega}-1) \widehat g(\omega)|^2 d\omega\\
	&\leq \liminf_{\tau\to 0^+}\int_\R |\tau^{-1} (e^{-i\tau \omega}-1) \widehat g(\omega)|^2 d\omega\\
		& = \liminf_{\tau\to 0^+}\Big\|\tau^{-1}(g(\cdot +\tau)- g(\cdot))\Big\|_{L^2(\R)}
		\end{aligned}
		$$
	for functions $g\in L^2(\R)$.
	Setting $g=\chi^{\eps} f$ with $\chi^\eps(z):=\chi(z/\eps)$ a smooth cutoff function,
$\chi(z)$ equal to zero for $z=0$ and $1$ for $z\geq 1$, 
and observing that $\sup |(\chi^\eps)'| \lesssim \eps^{-1}$,
we thus have
$$
\tau^{-1}\|\chi^{\eps} f(\cdot+\tau)- \chi^{\eps} f(\cdot)\|_{L^2(\R)}\lesssim
\tau^{-1} \|f(\cdot +\tau)-f(\cdot)\|_{L^2(0,+\infty)} +     
\eps^{-1} \|f \|_{L^2(0,+\infty)},
$$ 
whence $f\in H^1(\eps,+\infty)$, and $f\in C^0[\eps,+\infty)$, for each $\eps>0$
(albeit with bound $\sim \eps^{-1}$ blowing up as $\eps\to 0^+$).

In particular, $f$ has a well-defined vaue $f(\eps)$ at $x=\eps$.
Extending $f|_{[\eps,+\infty)}$ by 
$$
\tilde f(t):=\begin{cases}
f(t)& t\geq\eps,\\
f(\eps)& t\leq \eps,
\end{cases}
$$
and defining a different smooth cutoff function $\tilde \chi^L(z):=\tilde \chi(z /L)$,
	with $\tilde \chi(z)$ equal to zero for $z\leq -1$ and $1$ for $z\geq 0$, set 
	$\tilde g=\tilde \chi^L \tilde f$.
Using $\tilde g'= (\tilde \chi^L)'f(\eps) + \tilde \chi^L f'|_{(\eps, +\infty)}$, we have
	\be\label{onee} 
	\|\tilde f'\|_{L^2(\eps,+\infty)}
	\leq
	\|\tilde g'\|_{L^2(\R)}\leq \liminf_{\tau\to 0^+} \tau^{-1}\|\tilde g(\cdot+\tau)-\tilde g\|_{L^2(\R)}.
	\ee

	Computing, by Jensen's inequality that 
	$$
	\begin{aligned}
		\tau^{-2}\int_0^\tau |f(\eps+t)-f(\eps)|^2 dt &= \int_0^\tau \Big|\int_0^t  f'(\eps+s) ds\Big|^2 dt\\
		&\leq
		\tau^{-2} \int_0^\tau t \, dt  \|f'\|_{L^2(\eps, \eps+\tau)}^2
		\leq   \|f'\|_{L^2(\eps, \eps+\tau)}^2\to 0
	\end{aligned}
	$$
	 as $\tau\to 0^+$, 
	and using $\|(\tilde \chi^L)'\|_{L^2(\R)}=L^{-1/2} \|\tilde \chi'\|_{L^2(\R)}$ together
	with the reverse inequality
	$
	\tau^{-1}|\|\tilde \chi^L(\cdot+\tau)-\tilde \chi^L\|_{L^2(-\infty,\eps-\tau)}
	\leq \| (\tilde \chi^L)' \|_{L^2(-\infty,\eps)}$,
	 we have also, for $\tau<\eps$,
\ba\label{twoe}
\tau^{-1}\|\tilde g(\cdot+\tau)-\tilde g\|_{L^2(\R)} \, -  \, 
&\tau^{-1}\|f(\cdot+\tau)-f\|_{L^2(\eps, \infty)} \leq \\
	&
	\tau^{-1}|f(\eps)| \, \|\tilde \chi^L(\cdot+\tau)-\tilde \chi^L\|_{L^2(-\infty,\eps-\tau)} \\
	&\quad +
	\tau^{-1} \Big(\int_0^\tau |f(\eps+t)-f(\eps)|^2 dt\Big)^{1/2} 
	\to 0
	\ea
as $L\to \infty$ and $\tau\to 0^+$.

Taking $\liminf_{\tau\to 0^+}\lim_{L\to \infty}$ of \eqref{twoe}, and combining with \eqref{onee},
we thus obtain
$$
\|f'\|_{L^2(\eps,+\infty)}\leq 
	\liminf_{\tau \to 0^+} \tau^{-1} \|f(\cdot +\tau)-f(\cdot)\|_{L^2(\eps,+\infty)}
$$
for each $\eps>0$, whence the first assertion follows in the limit as $\eps\to 0^+$.

(ii) Similarly, for functions $g\in H^1(\R)$, we have
$\tau^{-1} (g(\cdot +\tau)-g(\cdot))\to g'$ in $L^2(\R)$, by
$\lim_{\tau\to 0^+} 
\tau^{-1} (e^{-i\tau \omega}-1)\widehat g = i\omega \widehat g(\omega)$ a.e.
and the Lebesgue dominated convergence theorem, noting that $|\tau^{-1} (e^{-i\tau \omega}-1)\widehat g |\lesssim
|\omega \widehat g| \in L^2(\R)$ by $\sup_{z\in \R } |z^{-1}(e^{iz}-1)|\lesssim 1$.
Thus, for $f\in H^1(t_0-\delta, t_0+\delta)$, $\delta>0$, we have, defining $g:=\chi (t)f(t)$,
where $\chi$ is a smooth cutoff equal to $1$ for $t\in (t_0-\delta/2, t_0+\delta/2)$ and zero outside
$(t_0-\delta,t_0+\delta)$, that $\tau^{-1} (g(\cdot +\tau)-g(\cdot))\to g'$ in $L^2(\R)$, and therefore
$\tau^{-1} (f(\cdot +\tau)-f(\cdot))\to f'$ in $L^2( t_0-\delta/4, t_0+\delta/4)$.
As $t_0$ and $\delta>0$ were arbitrary, the result follows.
\end{proof}

\bl\label{ldiff2}
For any $t\in \R$, there holds
$$
\int_t^{+\infty}|x|^4 \leq 
\Big( \frac{2}{2^{1/4}-1}\Big)^2
\Big(\int_t^{+\infty}|x|^2\Big)
\Big( \sup_{\tau>0} \int_t^{+\infty} \tau^{-1} |x(s+\tau)-x(s)|^2 ds \Big).
$$
\el

\begin{proof}
Denoting $C_1=\int_t^{+\infty}|x|^2$ and $C_2= \sup_{\tau>0} \int_t^{+\infty}\tau^{-1}|x(s+\tau)-x(s)|^2 ds$,
we have 
$$
\int_t^{+\infty}|x|^2\leq C_1
$$
and $\int_t^{+\infty}|x(s+\tau)-x(s)|^2 ds\leq C_2 \tau$ for all $\tau>0$.

Let 
$$
\phi_h(t):=
	\begin{cases}
		1/h & \hbox{\rm for $0\leq t\leq h$},\\
		0 & \hbox{\rm otherwise.}
	\end{cases}
$$
Then, on $(t,+\infty)$, we have for any choice of scale $\sigma>0$ the Haar decomposition
$$
\begin{aligned}
	x(s)&=
	\lim_{k\to \infty} \int_0^{+\infty}x(s+u)\phi_{2^{-k}\sigma}(u) du \\
	&= \int_0^{+\infty}x(s+u)\phi_\sigma (u) du +
\sum_{k\geq 1,\,  h=2^{-k}\sigma} \int_0^{+\infty}x(s+u)[\phi_h(u)-\phi_{2h}(u)] du \\
&=:
	x_\sigma +\sum_{k\geq 1,\,  h=2^{-k}\sigma} \tilde x_{h} .
\end{aligned}
$$

Note by Jensen's inequality that 
	$$
\begin{aligned}
	\int_t^{+\infty}|x_\sigma|^2 & \leq 
	\int_t^{+\infty} \int_0^{+\infty} \phi_\sigma(u) |x(s+u)|^2 du\, ds\\
	&= \int_0^{+\infty}\phi_\sigma(u) \int_t^{+\infty}  |x(s+u)|^2 ds\, du\\
	&\leq \Big(\int_0^{+\infty} \phi_\sigma\Big) \Big( \int_t^{+\infty}|x|^2\Big)\leq C_1,
\end{aligned}
	$$
	and by Cauchy-Schwarz' inequality that
$$
\|x_\sigma\|_{L^\infty(t,+\infty)}\leq
\Big(\int_t^{+\infty}|x|^2\Big)^{1/2}
	\Big(\int_0^{+\infty}|\phi_\sigma|^2\Big)^{1/2}\leq \sqrt {C_1} (1/\sqrt \sigma)=\sqrt{C_1/\sigma},
$$
whence
$\int_t^{+\infty}|x_\sigma|^4\leq 
\|x_\sigma\|_{L^\infty(t,+\infty)}^2 \int_t^{+\infty}|x_\sigma|^2\leq C_1^2/\sigma.$

Note also that $\phi_h(u)-\phi_{2h}(u)= (1/2)[\phi_h(u)-\phi_h(u-h)]$, so that
$$
\begin{aligned}
	\tilde x_h(s)&= (1/2)\int_0^{+\infty} x(s+u)[\phi_h(u)-\phi_h(u-h)]du\\
	& = (1/2)\int_0^{+\infty} [x(s+u)-x(s+u+h)]\phi_h(u) du,
\end{aligned}
$$
and therefore, similarly to the previous computation,
$$
	\int_t^{+\infty}|\tilde x_h|^2\leq (1/4)\int_t^{+\infty} |x(s+h)-x(s)|^2 ds\leq (1/4)C_2 h,
$$
while
$$
\begin{aligned}
	\|\tilde x_h\|_{L^\infty(t,+\infty)}&\leq
(1/2)\sqrt{ \int_t^{+\infty} |x(s+h)-x(s)|^2 ds}\sqrt{\int_0^{+\infty} \phi_h^2}\\
	& \leq (1/2)\sqrt{C_2h} (1/\sqrt{h})=(1/2)\sqrt{C_2}.
\end{aligned}
$$

Thus, 
$\int_t^{+\infty} |\tilde x_h|^4 \leq \|\tilde x_h\|_{L^\infty(t,+\infty)}^2 \int_t^{+\infty}|\tilde x_h|^2\leq (1/16)C_2^2 h\leq (1/16)C_2^2 \sigma 2^{-k}$ 
for $h=\sigma2^{-k}$, whence
$$
\begin{aligned}
\|x\|_{L^4(t,+\infty)} & \leq
\|x_\sigma \|_{L^4(t,+\infty)} +\sum_{k\geq 1,\,  h=2^{-k}\sigma} \|\tilde x_{h} \|_{L^4(t,+\infty)} \\
&\leq (C_1^2/\sigma)^{1/4} + (C_2^2\sigma/16)^{1/4} \sum_{k\geq 1} 2^{-k/4} \\
	&= C_1^{1/2}/\sigma^{1/4} + C_2^{1/2} \sigma^{1/4}/(2(2^{1/4}-1)). 
\end{aligned}
$$
Taking $\sigma^{1/4} =(C_1/C_2)^{1/4}\sqrt{2(2^{1/4}-1)}$ to minimize the right-hand side
gives 
$$
\|x\|_{L^4(t,+\infty)} \leq \Big((2/(2^{1/4}-1))^2 C_1C_2\Big)^{1/4} $$
as required.
\end{proof}

\section{Integrals and counterexample}\label{s:bochner}
We conclude with an explicit counterexample showing that our expression \eqref{greens} of the 
variation of constants formula in terms of improper integrals as $a\to 0^+$ cannot in general 
be reformulated in terms of standard Bochner integrals for Hilbert space-valued functions \cite{DU}.

Let $\bH= L^2(0,1)$ and $A$ the multiplication operator
$Ah(\alpha)=\alpha h(\alpha)$. Then, $T_s(t)$ is also a multiplication operator, with
$T_s(t)h(\alpha)= e^{-t/\alpha} h(\alpha)$ and 
$A^{-1}T_s(t)h(\alpha)= \alpha^{-1}e^{-t/\alpha} h(\alpha)$.
Meanwhile $\Pi_s=\Id$, $\Pi_u$, $T_u$, and $\Pi_c$
are identically zero, hence \eqref{greens} reduces to 
\be\label{boch}
	y(t)=y_s(t)=\lim_{a\to 0^+}\int_{-\infty}^t A^{-1}T_s(t-\tau)f^a(\tau)d\tau.
	\ee
The question is whether the righthand side of the above, or, equivalently, after the 
change of variables $\sigma:=t-\tau$, expression
$ \lim_{a\to 0^+}\int_0^{+\infty} A^{-1}T_s(\sigma)f^a(t-\sigma) d\sigma $,
can be interpreted as a Bochner integral $\int_0^{+\infty} A^{-1}T_s(\sigma)f(\cdot-\sigma) d\sigma$
in $L^2(\R, \bH)$,
that is, as an integral over $\sigma \in (0,+\infty)$ of an integrand valued not in $\bH$ but
in the large Hilbert space $L^2(\R,\bH)$.

Recall \cite{DU} that Bochner integrability requires Lebesgue integrability 
with respect to $\sigma$ of the norm of the integrand, in this case, integrability on $(0,+\infty)$ of
\be\label{norm}
\|A^{-1} T_s(\sigma)f(\cdot-\sigma )\|_{L^2(\R,\bH)}
=
\Big(\int_\R \int_0^1 \big(\alpha^{-1}e^{-\sigma/\alpha}f(t-\sigma,\alpha)\big)^2 d\alpha \, dt \Big)^{1/2}.
\ee

Take now $f(t, \alpha):= \alpha^{-1/2} (|\log \alpha|+1)^{-r/2} \phi(t)$, where $\phi\in L^2$ and
$r>1$. Then, 
$$
	\|f\|_{ L^2(\R,\bH)}^2= \int_\R \int_0^1 \big(\alpha^{-1/2} 
	(|\log \alpha|+1)^{-r/2} \phi(t)\big)^2 d\alpha \, dt
	= \|\phi\|_{L^2}^2 \int_0^1 \alpha^{-1} (|\log \alpha|+1)^{-r} d\alpha\big)
$$
is finite, hence $f \in L^2(\R, \bH)$.
But, 
$$
\begin{aligned}
	\Big(\int_\R \int_0^1 \big(\alpha^{-1}&e^{-\sigma/\alpha}f(t-\sigma,\alpha)\big)^2 d\alpha \, dt 
	\Big)^{1/2}=\\
	&\Big(
	\int_\R \int_0^1 \big(\alpha^{-1}e^{-\sigma/\alpha}\alpha^{-1/2} (|\log \alpha|+1)^{-r/2} \phi(t)\big)^2 d\alpha \, dt
	\Big)^{1/2}\\
	&=
	\|\phi\|_{L^2(\R,\R)}\Big(\int_0^1 \alpha^{-3} e^{-2\sigma/\alpha} (|\log \alpha|+1)^{-r} d\alpha\Big)^{1/2},
\end{aligned}
$$
after the change of coordinates $z:=\alpha/\sigma$ becomes
$$
\begin{aligned}
\|\phi\|_{L^2(\R,\R)} \sigma^{-1}   \,
	\Big(\int_0^{1/\sigma}  \frac{  z^{-3} e^{-2/z} }{(|\log z + \log \sigma|+1)^r}  dz \Big)^{1/2}
	&\gtrsim \sigma^{-1} \Big(\int_{1/2}^1   \frac{z^{-3} e^{-2/z} }{(|\log z + \log \sigma| +1)^r}  dz \Big)^{1/2}\\
	&\gtrsim \sigma^{-1} ( |\log \sigma|+1)^{-r/2}
\end{aligned}
$$
as $\sigma\to 0$, which, for $1<r\leq 2$ is not integrable on $(0,+\infty)$.
This example shows that \eqref{greens} cannot in general be expressed in terms of 
standard, Bochner integrals.

\end{document}